\documentclass[a4paper,12pt]{amsart}

\pdfoutput=1
\usepackage{graphicx}
\usepackage{hyphenat}
\usepackage[english]{babel}
\usepackage[utf8]{inputenc}
\usepackage[T1]{fontenc}
\usepackage{amsfonts}
\usepackage{amsmath}
\usepackage{amssymb}
\usepackage{amsthm}
\usepackage{faktor} 
\usepackage{bbm} 
\usepackage{yhmath} 
\usepackage{stmaryrd} 
\usepackage{empheq}
\usepackage{tikz-cd}
\usepackage{verbatim}
\usepackage{cite}
\usepackage{url}
\usepackage{relsize}
\usepackage{hyperref}
\usepackage{stackrel} 
\usepackage{leftidx}
\usepackage[foot]{amsaddr}
\usepackage{theoremref}

\usetikzlibrary{decorations.markings}
\newcommand{\notarrow}{\ensuremath{\vartimes}}
\newcommand{\Longdoesnotimply}{\ensuremath{\mathbin{=\!\!=\!\!\!\raisebox{-.1em}{\notarrow}}}}
\newcommand{\Doesnotimply}{\ensuremath{\mathbin{=\!\!\!\raisebox{-.1em}{\notarrow}}}}
\tikzset{notsrc/.style={
        decoration={markings,
            mark= at position 0 with {
                \node[transform shape] (tempnode) {\notarrow};
            }
        },
        postaction={decorate}
    }
}
\tikzset{notdst/.style={
        decoration={markings,
            mark= at position 1 with {
                \node[transform shape] (tempnode) {\notarrow};
            }
        },
        postaction={decorate}
    }
}

\newcommand{\txt}{\mathrm} 
\newcommand{\txtt}{\textrm} 

\renewcommand{\th}{\ensuremath{^\textrm{\scriptsize th}}}
\newcommand{\st}{\ensuremath{^\textrm{\scriptsize st}}}

\newcommand{\setF}{\mathbb{F}}
\newcommand{\setR}{\mathbb{R}}
\newcommand{\setRppos}{\mathbb{R}_{> 0}}

\newcommand{\setC}{\mathbb{C}}
\newcommand{\setZ}{\mathbb{Z}}
\newcommand{\setN}{\mathbb{N}}
\newcommand{\setNz}{\mathbb{N}_0}

\newcommand{\scf}[1]{\underline{#1}} 

\renewcommand{\subset}{\subseteq}

\newcommand{\iso}{\cong}
\newcommand{\isomapsto}{\stackrel{\sim}{\longmapsto}}
\renewcommand{\mapsfrom}{\mathrel{\reflectbox{\ensuremath{\mapsto}}}}

\newcommand{\colim}{\txt{colim}} 
\newcommand{\diag}{\txt{diag}}
\newcommand{\diagprod}{\mathbin{\times_\diag}} 
\newcommand{\pr}{\txt{pr}} 
\newcommand{\projvect}{\txt{P}} 
\newcommand{\ev}{\txt{ev}} 
\newcommand{\sphere}{\txt{S}} 
\newcommand{\scale}{\txt{sc}} 
\newcommand{\conj}[1]{\overline{#1}}
\newcommand{\ind}{\mathbbm{1}} 
\newcommand{\Sp}[1]{\txt{Sp}{#1}} 
\newcommand{\e}{\txt{e}} 
\newcommand{\Tg}{\txt{T}} 
\newcommand{\B}{\txt{B}} 
\renewcommand{\i}{\txt{i}} 
\newcommand{\D}{\txt{D}} 
\renewcommand{\d}{\txt{d}} 
\renewcommand{\Re}{\txt{Re}} 
\renewcommand{\Im}{\txt{Im}}
\newcommand{\id}{\txt{id}} 
\newcommand{\eqdef}{:=}
\newcommand{\restr}[1]{ |_{#1} } 

\theoremstyle{plain}                    
\newtheorem{theorem}{Theorem}[section]
\newtheorem{proposition}[theorem]{Proposition}
\newtheorem{lemma}[theorem]{Lemma}
\newtheorem{corollary}[theorem]{Corollary}

\theoremstyle{definition}
\newtheorem{definition}[theorem]{Definition}
\newtheorem{example}[theorem]{Example}

\theoremstyle{remark}
\newtheorem{remark}[theorem]{Remark}

\title{Hamiltonian Partial Differential Equations and Symplectic Scale Manifolds}
\author{João Bernardo Crespo}
\email{scrovy@yahoo.com}
\author{Oliver Fabert}
\email{oliver.fabert@gmail.com}
\address{Department of Mathematics, VU Amsterdam, The Netherlands.}
\begin{document}

\begin{abstract}
This paper defines symplectic scale manifolds based on Hofer-Wysocki-Zehnder's scale calculus. We introduce Hamiltonian vector fields and flows on these by narrowing down sc-smoothness to what we denote by \emph{strong} sc-smoothness, a concept which effectively formalizes the desired smoothness properties for Hamiltonian functions. We show the concept to be invariant under sc-smooth symplectomorphisms, whence it is compatible with Hofer's scale manifolds. We develop and verify the theory at the hand of the free Schrödinger equation.
\end{abstract}

\maketitle
\tableofcontents
\markboth{J.B.~Crespo \and O.~Fabert}{Hamiltonian Partial Differential Equations and Symplectic Scale Manifolds} 

\section{Introduction} \label{sec:intro}

Hamiltonian partial differential equations (PDEs) have received increasing attention in the last forty years. A select number of examples include the Schrödinger, Korteweg\hyp{}de~Vries and the Boussinesq equations. These PDEs are intrinsically linked to infinite\hyp{}dimensional symplectic geometry: their evolution is typically analysed on an infinite\hyp{}dimensional space, and their solutions can heuristically be expressed as integral curves of a vector field obtained by means of a Hamiltonian function and a symplectic structure. In fact, the link between Hamiltonian PDEs and infinite\hyp{}dimensional symplectic geometry is, in some sense, akin to the one between the well-known Hamilton's equations
\begin{empheq}[left = \empheqlbrace, right = \mbox{$\,,~k = 1,2,\ldots,d\,.$}]{align} \label{eq:hamilton_n_particles}
	\dot q_k &= \frac{\partial h}{\partial p_k} \nonumber \\
	\dot p_k &= -\frac{\partial h}{\partial q_k}
\end{empheq}
of classical mechanics and \emph{finite}\hyp{}dimensional symplectic geometry.

\subsection{Finite-dimensional Symplectic Geometry} \label{sec:finite_dim}
Before delving in\-to the expectedly more involved case of infinite dimensions, we start by reviewing finite\hyp{}dimensional symplectic geometry. A real vector space $V$ of dimension $2d$ is said to be symplectic whenever adjoined with a bilinear skew\hyp{}symmetric form $\omega: V \times V \mapsto \setR$ which, similarly to an inner product, identifies $V$ with its dual by means of the isomorphism of vector spaces $\iota_{\omega}: V \isomapsto V^*, \, v \mapsto \omega(\cdot, v)$. The canonical example to have in mind is the coordinate space $\setR^{2d} = \setC^{d}$ with its standard symplectic form
\begin{equation} \label{eq:standard_symplectic_form}
	\omega(v,w) = \langle \i v, w \rangle \,,
\end{equation}
where $\i: \setC^{d} \mapsto \setC^{d}, \,v \mapsto \i v$ is its standard complex structure and $\langle \cdot, \cdot \rangle$ its standard real inner product.

More generally, a smooth manifold $M$ of dimension $2d$ is said to be symplectic whenever a maximal smooth atlas is available with symplectomorphisms as transition maps. This means that for each pair of coordinate charts $\phi: U_\phi \subset M \isomapsto V_\phi \subset \setR^{2d}$ and $\psi: U_\psi \isomapsto V_\psi$, we require that the derivative $T_{\phi,\psi}(x) \eqdef \d_x (\psi \phi^{-1}): \setR^{2d} \isomapsto \setR^{2d}$ preserves the standard symplectic form for all $x \in \phi(U_\phi \cap U_\psi)$, in the sense that $\omega(T_{\phi,\psi}(x)\cdot v, T_{\phi,\psi}(x)\cdot w) = \omega(v,w)$ for all $v,w\in\setR^{2d}$.  An equivalent definition due to Darboux \cite[Theorem~3.15]{mcduff98} is that of a manifold $M$ together with a closed two-form which is comprised of a symplectic form $\omega_p: \Tg_p M \times \Tg_p M \mapsto \setR$ on each tangent space, $p\in M$. Naturally, a symplectic vector space is also a symplectic manifold with its symplectic form at every point. The study of symplectic manifolds is the topic of symplectic geometry, and a sound introduction may be found in \cite{mcduff98}.

The link between Hamilton's equations and finite\hyp{}dimensional symplectic geometry is settled down in the following. Any smooth function $h: M \mapsto \setR$ on a symplectic manifold $M$ gives rise to a Hamiltonian vector field $V_h: M \mapsto \Tg M$ uniquely defined by the relation
\begin{equation} \label{eq:symplectic_gradient_finitedim_mfld}
	-\d_p h = \omega_p(\cdot, V_h(p)) \in \Tg_p^* M ~~\txtt{for all}~p\in M \,,
\end{equation}
or compactly written using the isomorphism of vector bundles $\iota_\omega: \Tg M \isomapsto \Tg^* M$ 
induced by the symplectic structure of $M$,
\begin{equation} \label{eq:symplectic_gradient2_finitedim_mfld}
	-\d h = \iota_\omega \circ V_h: M \mapsto \Tg^* M \,.
\end{equation}
The flow $\varphi_h$ of the vector field $V_h$ is then itself said to be Hamiltonian, and both are said to be generated by $h$. Returning to our canonical example $M = \setR^{2d}$ with its standard form, simple computations show that the integral curves of the Hamiltonian vector field generated by a smooth Hamiltonian function $h: \setR^{2d} \mapsto \setR$ are precisely the solutions of Hamilton's equations \eqref{eq:hamilton_n_particles} if we label the coordinates of $\setR^{2d} = \setR^{d} \oplus \setR^{d}$ as $x = (q,p)$. Accordingly, for a general manifold, the integral curves of $V_h$ are locally given as solutions of Hamilton's equations as well.

Though at first sight this geometrical approach might seem an unnecessary mathematical artefact to study these equations of classical mechanics, symplectic manifolds arise naturally in the study of Hamiltonian functions with symmetries. Consider, for instance, the smooth action of the circle $S^1 = \{\lambda \in \setC: |\lambda| = 1\}$ on $\setC^d$ by pointwise multiplication and an $S^1$\hyp{}invariant Hamiltonian function $h: \setC^d \mapsto \setR$, such as the one used in a finite\hyp{}dimensional quantum mechanical system \cite{ashtekar97}. One can check that the fundamental vector field of the action is generated by the $S^1$\hyp{}invariant Hamiltonian function (momentum map) $\mu: \setC^d \mapsto \setR$
\begin{equation}
	\mu(x) = \frac{1}{2}(1-\|x\|^2)
\end{equation}
and that the action is free on $\mu^{-1}(0)$. By the Marsden\hyp{}Weinstein symplectic reduction theorem \cite[Theorem~4.3.1]{weinstein87}, one obtains a natural symplectic manifold structure on
\begin{equation}
	\setC\projvect{}^{d-1} \eqdef \faktor{\setC^d \setminus \{0\}}{\setC^*} \iso \faktor{\mu^{-1}(0)}{S^1} \,.
\end{equation}
Furthermore, the Hamiltonian function $h$ and flow $\varphi_h$ descend to functions $\bar h$ and $\bar \varphi_h$ on the quotient, respectively, and $\bar \varphi_h$ is precisely the Hamiltonian flow generated by $\bar h$ \cite[Theorem~4.3.5]{weinstein87}. One can then study the reduced system on the lower\hyp{}dimensional manifold $\setC\projvect{}^{d-1}$ and recover the original dynamics therefrom \cite[pp.~304--305]{weinstein87}.

\subsection{Hamiltonian Partial Differential Equations}
Similarly to the finite\hyp{}dimensional case, the solutions of Hamiltonian PDEs can be considered to be integral curves of a Hamiltonian vector field, but infinite\hyp{}dimensional symplectic vector spaces are needed instead. As an example, take the free Schrödinger equation on the circle
\begin{equation} \label{eq:free_schrodinger_intro}
	\i u_t = - \Delta u
\end{equation}
for an unknown wave function $u: \setR \times S^1 \mapsto \setC, \, (t,x) \mapsto u(t,x)$, where $\Delta u = u_{xx}$ is the Laplacian operator. To study this equation, introduce the Hilbert space of square\hyp{}integrable functions on the circle
\begin{equation}
		L^2 \eqdef L^2(S^1, \setC) = \Big\{u: S^1 \mapsto \setC: u~\text{measurable and}~\int_{S^1} |u|^2 < \infty \Big\}
\end{equation}
and its standard symplectic form $\omega: L^2 \times L^2 \mapsto \setR$ given by \eqref{eq:standard_symplectic_form}, where now the real inner product and complex structure are the ones of $L^2$.

As in Section~\ref{sec:finite_dim}, we try and define Hamiltonian vector fields for each Hamiltonian map $h$ by requiring
\begin{equation} \label{eq:omega_gradient_hilbert_intro}
	- \d h(u) = \omega(\cdot, V_h(u))
\end{equation}
for an adequate set of functions $\{u\}$. With this idea in mind, we observe that the solutions of \eqref{eq:free_schrodinger_intro} are integral curves of the Hamiltonian vector field generated by the Hamiltonian function
\begin{equation} \label{eq:hamiltonian_freeschrodinger_intro}
	h(u) = \frac{1}{2} \int_{S^1} |u_x(a)|^2 \,\d a \,.
\end{equation}
Indeed, integration by parts shows that $\d h(u) \cdot v = \langle u_x, v_x \rangle = - \langle u_{xx}, v \rangle$, whence $V_h(u) = \i u_{xx}$. The corresponding Hamiltonian flow is
\begin{equation} \label{eq:flow_freeschrodinger_intro}
	\varphi_h(t, u) = \e^{\i t \Delta} u \,.
\end{equation}
Further examples and ellaboration on Hamiltonian PDEs can be found in \cite{kuksin00,craig08,bourgain10,abbondandolo14,miller06,bona88}.

Although the presented setup seems plausible, it shows a crucial difference with respect to the finite dimensional case. If we inspect the proposed mathematical objects, we see that the Hamiltonian function \eqref{eq:hamiltonian_freeschrodinger_intro} cannot be defined on the entire space $L^2$, but only on the dense subset $W^{1,2} = W^{1,2}(S^1,\setC)$ of weakly differentiable functions with $L^2$ derivative. Similarly, the vector field $V_h$ is only densely\hyp{}defined and two derivatives are needed. Rather in contrast to this, the Hamiltonian flow $\varphi_h$ defines a map $\setR \times L^2 \mapsto L^2$. We thus recognize that several vector spaces are needed for defining the different objects at stake.

This nuance was elegantly solved by Kuksin \cite{kuksin00}, who used Hilbert scales to frame Hamiltonian PDEs. A Hilbert (Banach) scale is a filtration of Hilbert (Banach) spaces which are densely and compactly embedded into each other. From $L^2$, we can build the Levi\hyp{}Sobolev Hilbert scale $\{W^{k,2}\}_{k\in\setZ}$ with $W^{0,2} = L^2$ and extend the real inner product of $L^2$, hence also $\omega$, to a non\hyp{}degenerate pairing $W^{k,2} \times W^{-k,2} \mapsto \setR$. Since \eqref{eq:hamiltonian_freeschrodinger_intro} defines a (Frèchet) smooth map $h: W^{1,2} \mapsto \setR$, the usual $\omega$-gradient relation \eqref{eq:omega_gradient_hilbert_intro} produces a vector field $V_h: W^{1,2} \mapsto W^{-1,2}$ which is simply $\i \Delta$ --- Kuksin's framework involving Hilbert scales delivers the expected results.

Moving one step further, suppose that we are only interested in nonzero wave functions of Schrödinger's equation up to a nonzero complex scalar. This is the case of interest in physics, where the equivalence classes in the projective Hilbert space
\begin{equation}
	\projvect(L^2) \eqdef \faktor{L^2 \setminus \{0\}}{\setC^*}
\end{equation}
represent the state of the quantum\hyp{}mechanical system \cite{ashtekar97,fabert17} (again, $\setC$ acts on $L^2$ by pointwise multiplication). To describe such a system we desire, by analogy with the finite\hyp{}dimensional case, to have some local symplectic scale structure on $\projvect(L^2)$ where we can make sense of basic symplectic geometry as in Section~\ref{sec:finite_dim}.

In our path towards this aim, part of the work by Hofer, Wysocki and Zehnder on polyfolds is essential \cite{hofer06,hofer08_polyfolds,hofer10,hofer17}. Departing from a Banach scale, they develop the notion of scale calculus, which allows the derivative of a function between scales to be defined only on a dense subspace of higher regularity. Subsequently, they extend scale differentiation inductively to scale smoothness, where an arbitrary number of scale derivatives may be taken. In a similar way as in classical differential geometry \cite{lang02},\cite[Chapter~73]{zeidler97}, the authors then proceed to introduce smooth scale manifolds (also notated $\scale^\infty$\hyp{}manifolds) locally modeled on Banach scales.

As expected, scale manifolds inherit structures from the underlying local model. Specifically, an $\scale^\infty$\hyp{}manifold $M$ gives rise to a natural filtration $\{M_k\}_{k\in\setNz}$ induced by the local scale structure, and for each point of a filtration subspace $M_k$, $k\geq 1$, we can associate a partial Banach scale which plays the role of the tangent space. Also, a tangent bundle $\pi_{\Tg M}: \Tg M \mapsto M^1$ is defined, where $M^1 = \{M_{k+1}\}_{k\in\setNz}$ is the shifted filtration. The shift appearing in the base space of the bundle reflects the higher regularity of the differentiation points.

Scale smoothness and polyfolds were originally introduced with the purpose of solving problems in symplectic field theory and related areas \cite{fabert10,hofer11}. Fabert et al.~\cite{fabert12} review the theory in a broader perspective, also extending the idea of Banach scales to filtrations of topological spaces. Wehrheim \cite{wehrheim12} elaborates on the Hofer\hyp{}Wysocki\hyp{}Zehnder Fredholm theory of \cite{hofer06,hofer07,hofer08_fredholm,hofer10}, while noting that a Banach scale can be recovered from its (Frèchet) limit and restricted norms. Gerstenberger \cite{gerstenberger16} works with the limits of Banach scales as well, modifying the scale smoothness and Fredholm theories of Hofer and Wehrheim so as to allow for the application of the Nash\hyp{}Moser theorem on ``tame'' Frèchet limits.

We choose scale smoothness and manifolds to handle Hamiltonian PDEs since, as motivated above, Hamiltonian functions and vector fields are expected to be defined on points of higher regularity compared to the flow. In fact, for our guiding example of the free Schrödinger equation, informally differentiating the flow with respect to the time variable delivers
\begin{equation}
	\frac{\d}{\d t} \e^{\i t \Delta} u = \i \Delta \e^{\i t \Delta} u = V_h \circ \e^{\i t \Delta} u
\end{equation}
which exists as an element of $L^2$ whenever $u\in W^{2,2}$. Whereas the more classical formulation devised by Kuksin works with Frèchet smoothness, using scales only to handle the symplectic structure, Hofer\hyp{}Wysocki\hyp{}Zehnder have a native approach, incorporating the scale structure in their definition of smoothness. What is more, scale manifolds suit our example well, as the projective Hilbert space $\projvect(L^2)$ can be given such an $\scale$\hyp{}smooth structure (essentially in the same way as $\setC\projvect{}^d$).

\subsection{Main Contribution}

The work of Hofer\hyp{}Wysocki\hyp{}Zehnder allows us to generalize Banach scales to the manifold context, and it is not difficult to carry the generalization through to $\scale$\hyp{}smooth vector fields and flows: for example, a vector field is simply an $\scale$\hyp{}smooth section of the tangent bundle $V: M^1 \mapsto \Tg M$. Nevertheless, their work lacks a notion of a symplectic structure on an $\scale$\hyp{}smooth manifold. Naturally, with the lack of symplectic structures comes the lack of Hamiltonian vector fields and flows. Furthermore, it is not clear in what sense a Hamiltonian function should be smooth so as to obtain an $\scale$\hyp{}smooth vector field via a suitable symplectic gradient relation.

To fill this gap, we propose to define symplectic scale manifolds as scale manifolds locally modeled on a symplectic Banach scale, endowed with a maximal atlas of coordinate charts where all transition maps are symplectomorphisms. The latter condition allows for the definition of a cotangent bundle $\Tg^* M$ and a canonical isomorphism $\iota_\omega$ between tangent and cotangent bundles. Furthermore, we narrow down the scale smoothness concept of Hofer\hyp{}Wysocki\hyp{}Zehnder to what we baptize as \emph{strong} scale smoothness\footnote{Not to be confused with the side definition of \cite[Remark~1.3]{hofer17}.}. As it turns out, a desirable definition for the regularity of a Hamiltonian function requires the test vectors of the scale derivative to be taken from spaces of increasingly lower regularity as the regularity of the differentiation point increases, and the original $\scale$\hyp{}smoothness concept is too weak to accommodate this requirement. We prove that the concept of strong scale smoothness is invariant under pre\hyp{}composition with symplectomorphisms, hence it is consistent with symplectic $\scale$\hyp{}smooth manifolds.

The definition of strong scale smoothness leads to a natural generalization of Hamiltonian vector fields and flows in a symplectic $\scale$\hyp{}smooth manifold $M$: for a strongly $\scale$\hyp{}smooth function $h: M^1 \mapsto \setR$, we can interpret its derivative as an $\scale$\hyp{}smooth section of the cotangent bundle $\D h: M^1 \mapsto \Tg^* M$ and, as in the finite\hyp{}dimensional case of \eqref{eq:symplectic_gradient2_finitedim_mfld}, $\D h$ gives rise to an $\scale$\hyp{}smooth vector field $V_h$ by means of the bundle isomorphism $\iota_\omega: \Tg M \isomapsto \Tg^* M$ and the symplectic gradient relation
\begin{equation}
	-\D h = \iota_\omega \circ V_h : M^1 \mapsto \Tg^* M \,.
\end{equation}

The contributions of this work are developed and presented at the hand of the free Schrödinger equation \eqref{eq:free_schrodinger_intro} which, in the authors' modest opinion, is simple enough to avoid distractions and, at the same time, serves as a prototypical example exhibiting the core property of Hamiltonian PDEs: the vector field is only densely defined. Correspondingly, the projective Hilbert space $M = \projvect(L^2)$ is presented as a symplectic $\scale$\hyp{}smooth manifold locally modeled on the Hilbert scale $X = \{X_k = W^{2k,2}\}_{k\in\setNz}$. The flow $\varphi_h: \setR \times X \mapsto X$ of \eqref{eq:flow_freeschrodinger_intro} is shown to be $\scale$\hyp{}smooth and Hamiltonian, generated by the strongly $\scale$\hyp{}smooth Hamiltonian function $h: X^1 \mapsto \setR$ of \eqref{eq:hamiltonian_freeschrodinger_intro}. In the trend of symplectic reduction, these maps are subsequently seen to descend to maps $\bar \varphi_h: \setR \times M \mapsto M$ and $\bar h: M^1 \mapsto \setR$ inheriting the corresponding regularity properties, and $\bar \varphi_h$ is concluded to be a Hamiltonian flow generated by $\bar h$.

\subsection{Organization of Paper}
This paper is based on the first author's master thesis on the topic \cite{crespo17}, and the remainder is organized in three sections. Section~\ref{sec:scales} starts by reviewing linear scale structures. We introduce the concepts of Banach and Hilbert scales, showing how to build such a scale departing from a separable infinite\hyp{}dimensional Hilbert space. We also outline basic notions of linear symplectic geometry on scales. Subsequently, in Section~\ref{sec:hwz}, we present scale calculus by Hofer\hyp{}Wysocki\hyp{}Zehnder and corresponding scale manifolds. Finally, Section~\ref{sec:hampde} contains the main contribution of this paper. We extend the notions of vector fields and flows to scale manifolds and define strong scale smoothness on Banach scales, deducing its invariance under pre\hyp{}composition with symplectomorphisms. After this, we generalize the introduced concepts to the manifold setting. Throughout the section, we pair the developed theory with the guiding example of the free Schrödinger equation.

\subsection{General Notation}
In this paper, by convention, the natural numbers $\setN = \{1,2,3,\ldots\}$ start at one and we notate $\setNz = \{0\} \cup \setN = \{0, 1, 2, \ldots\}$ when starting at zero. Also, natural numbers can be considered as sets: $m = \{0, 1, \ldots, m-1\}$ for $m\in\setN$. For the Kronecker delta, we use $\delta_{mn}$. For sets $X$ and $Y$, $Y^X$ denotes the set of functions $X \mapsto Y$. Furthermore, we denote the strictly positive real numbers by $\setRppos$, the complex unit by $\i$, the real and imaginary parts of complex numbers by $\Re\{\cdot\}$ and $\Im\{\cdot\}$, respectively, and complex conjugation by $\conj{(\cdot)}$. Hermitian inner products are conjugate\hyp{}linear in the second argument. For a subset of some ambient space $U\subset X$, the notation $U(x)$ signifies $x\in U$. For the tangent bundle of a manifold $M$ at $p\in M$, we use $\Tg_p M$, and for the tangent and cotangent bundle, $\Tg M$ and $\Tg^* M$, respectively. The derivative of a map of manifolds $f: M \mapsto N$ at $p\in M$ is $\d_p f: \Tg_p M \mapsto \Tg_{f(p)} N$. For the time derivative of a curve $u: \setR \mapsto Z$, we use $\d u / \d t: \setR \mapsto Z$ or simply $\dot u$. The remaining notation is either standard or explained in the course of the paper.

\section{Scale Structures} \label{sec:scales}

In this section, we introduce Banach and Hilbert scales. We first define Banach scales within a framework which allows for natural operations such as translation of the index set, the product of scales, or a scale composed out of the topological duals of each space --- the ``dual scale''. Within the presented framework, we explore morphisms and several types of maps between scales. Subsequently, we narrow the focus down to Hilbert scales, where each underlying space has a compatible inner product. We show that a single separable Hilbert space induces a Hilbert scale prototypically modeled on weighted $l^2$ spaces. We also introduce symplectic structures on Banach and Hilbert scales.

\subsection{Banach Scales} \label{sec:banach_scales}

For introducing Banach scales, we use the basic categorical structure of a projective system. In this paper, we opt for a brief presentation, referring details to \cite{garrett10,leinster14}. Although Banach scales are frequently dealt with within a simpler framework of a filtration, when introducing dual scales, this simple framework will not suffice. Indeed, as we will see in this section, the dual scale of a filtration is not bond by inclusions but by adjoints of inclusions, and the explicit structure of a filtration disappears. Luckily, these adjoints will still be injective and it will still be useful to intuitively think of Banach scales as filtrations.

Consider the category of locally convex spaces (LCS) over a fixed field $\setF = \setR~\txt{or}~\setC$, with continuous linear maps as morphisms, and a non\hyp{}empty index set $S \subseteq \setZ$. A projective system of LCS on $S$ is a family $X \eqdef \{X_s\}_{s \in S}$ of LCS together with maps $p_{sr}: X_s \mapsto X_r$ for all $s>r \in S$, the so-called bonding maps, such that $p_{sq} = p_{rq} \circ p_{sr}$ for all $s>r>q \in S$. To such a projective system, we can assign a limit: a pair consisting of an LCS, $X_\infty \eqdef \lim_{s\in S} X_s$, and a collection of morphisms, $\{p_{\infty s}: X_\infty \mapsto X_s\}_{s\in S}$, such that $p_{\infty r} = p_{sr} \circ p_{\infty s}$ for all $s>r \in S$ and with the universal property that for any other such pair $(Y, \{f_s: Y \mapsto X_s\})$, there exists a a unique map $\phi: Y \mapsto X_\infty$ such that $f_s = p_{\infty s} \circ \phi$ for all $s \in S$. From this condition it follows that the limit is unique up to isomorphism, and an explicit expression is
\begin{equation} \label{eq:proj_limit_explicit}
	X_\infty = \left\{x=(x_s)_{s\in S} \in \prod_{s \in S} X_s: p_{sr}(x_s) = x_r~\txtt{for all}~s>r \right\}
\end{equation}
together with the canonical projections $p_{\infty s}: X_\infty \mapsto X_s, \, x \mapsto x_s$. Similarly, we can define the notion of a colimit $X_{-\infty} \eqdef \colim_{s\in S} X_s$ by reversing the arrows in the definition. Although $X_\infty$ is Hausdorff whenever all $X_s$ are so, the same is not compulsorily true for $X_{-\infty}$ \cite{floret80}.

The most important example of a projective system is a descending filtration of LCS, $\{X_s\}_{s\in S}$ with $X_s \subset X_r$ for $s>r$, bond by the canonical inclusions. In this case, the limit is $X_\infty = \bigcap_{s \in S} X_s$, colimit $X_{-\infty} = \bigcup_{s \in S} X_s$ and the corresponding limit maps are the inclusions as well. As in this particular case, the bonding maps will always be injective in this work, whence we notate them suggestively as $\iota_{sr}$ instead of $p_{sr}$. Due to this assumption, much of the set-theoretic machinery which is valid for filtrations remains valid in general.

In the following technical discussion, we bridge ubiquitous intuition, available when dealing with filtrations, with the more general case of an injective projective system. Basically, for any $s\in S\cup\{\infty\}$ and any $A\subset X_s$, we have isomorphic copies $\iota_{sr}(A) \subset X_r$ of $A$ for all $r<s$, whence it makes sense to define set-theoretical relations on subsets of different spaces. For $A\subset X_s$, $B\subset X_r$, $s>r$, and an arbitrary set $Z$, we employ the following conventions:
\begin{itemize}
	\item We say that $A \subset B$ if $\iota_{sr}(A) \subset B$ and identify $A$ with $\iota_{sr}(A)$.
	\item If $A\subset B$ and $f: B \mapsto Z$ is a function, we can restrict it to $A$, specifically $f\restr{A} \eqdef f \circ \iota_{sr}$. Dually, if we have $f: Z \mapsto A$ instead, we can embed the codomain in $B$.
	\item Switching the roles of the spaces, for $C\subset X_s$ and $D\subset X_r$, $s>r$, we have $D \subset C$ if $D \subset \iota_{sr}(C)$ and in that case, we identify $D$ with $\iota_{sr}^{-1}(D) \subset C$.
	\item Let now $A \subset B$ and $f: Z \mapsto B$. Then the codomain of $f$ restricts to $A$ if and only if $f(Z) \subset A$.
	\item We define $A\cap B \subset X_s$ by $A \cap \iota_{sr}^{-1}(B)$, which is isomorphic to $\iota_{sr}(A) \cap B \subset X_r$.
\end{itemize}

Still on an injective projective system $X$, we define a system of subsets on $X$, $A \subset X$, to be a collection of subsets $A = \{A_s\}_{s\in S}$ with $A_s \subset X_s$ on each $s\in S$, such that $A_s \subset A_r$ for all $s>r$ (as with the above conventions). It is called a system of open subsets if each $A_s$ is open in $X_s$. Such a collection allows us to restrict the projective system to subsets by restricting the bonding maps $\iota_{sr}: A_s \hookrightarrow A_r$. For two projective systems $X$ and $Y$ on $S$ and two systems of subsets $A\subset X$ and $B\subset Y$, we define the cartesian product $A \times B \eqdef \{A_s \times B_s\}_{s\in S}$. If the domains of $X$ and $Y$ differ, say the domains are $S$ and $S'$, respectively, we restrict them to the common domain $S\cap S'$ before taking the cartesian product. We also define, for $\tau\in\setZ$, the shifted system $A^\tau \eqdef \{A_{s+\tau}\}_{s\in S \cap (S-\tau)}$. Although the shifted system could be defined to take values on $S-\tau$, we do not use indices outside $S$ in this work.

Finally, before introducing Banach scales, we establish conventions on operator space topologies. Unless otherwise mentioned, we endow the set $\B(X, Y)$ of linear continuous maps between two Banach spaces $X$ and $Y$ with the operator norm $\|L\| \eqdef \sup_{x \in X, \|x\|_X \leq 1} \|L(x)\|_Y$. This norm makes $\B(X, Y)$ into a Banach space and its induced topology is called the strong or norm topology. Correspondingly, we denote by $X^* \eqdef \B(X, \setF)$ the topological dual of $X$ with operator norm. Yet, occasionally, it will be convenient to work with a weaker topology on $\B(X,Y)$ called the compact\hyp{}open topology. This is the topology with subbasis given by sets of the form $\{ L \in \B(X, Y): L(K) \subset U\}$, where $K\subset X$ is compact and $U\subset Y$ is open. The compact\hyp{}open topology on $\B(X,Y)$ has the property that for a metric space $A$ and a map $f: A \times X \mapsto Y$ with $f(a, \cdot) \in \B(X,Y)$ for all $a\in A$, $f$ is continuous if and only if the induced map $\bar f: A \mapsto \B(X,Y), \, a\mapsto f(a,\cdot)$ is continuous \cite{fox45}. If $\B(X,Y)$ has the strong topology, only the ``if'' part is valid in general.

After the technical framing of injective projective systems, we are now ready to define and explore Banach scales.
\begin{definition} \thlabel{def:banach_scales}
	A Banach scale on a non\hyp{}empty index set $S\subset\setZ$ (over $\setF$) is a projective system $X \eqdef \{X_s\}_{s \in S}$ of LCS where all $X_s$ are Banach spaces, the bonding maps $\iota_{sr}: X_s \hookrightarrow X_r$ are injective, $s>r \in S$,
	\begin{enumerate}
		\item the limit $X_\infty$ is dense in $X_s$ for all $s \in S$, and
		\item the bonding maps $\iota_{sr}$ are compact operators.
	\end{enumerate}
	The space $X_s$ is called the $s\th$ layer or level of $X$, $s\in S$. If each $X_s$ is completely normable but no specific norm is available, we call $X$ a Banachable scale. If $S=-S$ and $0\in S$, then $X_0$ is called the center of the scale\footnote{This is not to be confused with the center as defined in \cite{bonic67}, which is the limit of the scale instead.}. Also, for a property of a Banach(able) space $\mathcal{P}$, {\it e.g.}, reflexivity or separability, $X$ is said to have property $\mathcal{P}$ whenever all $X_s$ have property $\mathcal{P}$, $s\in S$.
\end{definition}

Compactness of the bonding maps is crucial in applications \cite{hofer06} and allows for a chain rule when we introduce calculus in this framework. We should also note that $X_\infty$ is a Frèchet space due to the (at most) countable cofinality of $S$ \cite{garrett13,schaefer71}. A trivial example of a scale is the constant scale $X_s = X$, $s \in S$ for a finite\hyp{}dimensional vector space $X$. In fact, this is the only possible scale if one of the spaces is finite\hyp{}dimensional, since finite\hyp{}dimensional subspaces of a normed space are closed. On the other hand, due to the compactness requirement, the same construction is not a scale if $X$ is infinite\hyp{}dimensional, unless $S$ is a singleton. The same argument shows that all scales $X = \{X_s\}_{s\in S}$ built out of infinite\hyp{}dimensional Banach spaces $X_s$ are proper, {\it i.e.}, $X_s \subsetneq X_r$ for all $s>r$.

We proceed with two important propositions which allow us to construct new scales departing from old ones by means of natural operations such as restriction, translation, products and taking duals of the individual Banach spaces. 
\begin{proposition} \thlabel{prop:restriction_scales}
	Let $X = \{X_s\}_{s \in S}$, $S \subseteq \setZ$, be a Banach scale and let $S' \subseteq S$ be a non\hyp{}empty subset. Then $X\restr{S'} \eqdef \{X_s\}_{s \in S'}$ is a Banach scale with bonding maps $\iota'_{sr} \eqdef \iota_{sr}$, $s>r \in S'$. Its limit $X'_\infty \eqdef \lim_{s \in S'} X_s$ is $X_\infty$ if $S'$ is unbounded above and $X_{\max S'}$ otherwise.
\end{proposition}
\begin{proof}
	It is clear that the projective system $\{X_s\}_{s \in S}$ restricts to an injective projective system on $S'$ with compact bonding maps. To compute the limit $X'_\infty$, we distinguish two cases. If $S'$ is unbounded above, then so is $S$. Consequently, $S'$ is cofinal in $S$, meaning that for each $r\in S$ there is $s \in S'$ such that $s\geq r$. This implies that the limits $X_\infty$ and $X'_\infty$ are uniquely isomorphic \cite{garrett10}, and the density claim follows immediately. On the other hand, if $S'$ is bounded above, then $\{\max S'\}$ is cofinal in $S'$. The limit is then $X_{\max S'}$ and density follows from $X_\infty \subset X_{\max S'}$ (via $\iota_{\infty, \max S'}$).
\end{proof}

\begin{proposition} \thlabel{prop:scale_constructions}
	Let $X \eqdef \{X_s\}_{s \in S}$ and $Y \eqdef \{Y_s\}_{s \in S}$ be Banach scales on $S\subset \setZ$ with (co-)limits $X_{\pm\infty}$, $Y_{\pm\infty}$ and maps $\iota^X_{sr}$ and $\iota^Y_{sr}$, $s>r \in S\cup\{\pm\infty\}$, respectively. We introduce the following constructions:
	\begin{enumerate}
		\item (shifted scale) For $\tau\in\setZ$, $X^\tau \eqdef \{X_{s+\tau}\}_{s \in S-\tau}$ is a Banach scale on $S-\tau = \{s-\tau: s \in S\}$ with bonding maps $\iota^X_{s+\tau,r+\tau}$, $s>r\in S-\tau$, and limit $(X_\infty, \{\iota^X_{\infty, s+\tau}\}_{s\in S-\tau})$. \label{prop:shifted_scale}
		\item (product of scales) $X \times Y \eqdef \{X_s \times Y_s\}_{s \in S}$ is a Banach scale on $S$ with bonding maps $\iota_{sr}^X \times \iota_{sr}^Y$, $s>r$, and limit $(X_\infty \times Y_\infty, \{\iota_{\infty s}^X \times \iota_{\infty s}^Y\}_{s\in S})$. \label{prop:direct_sum_scale}
		\item (dual scale) Endow $X_s^*$ with the operator norm, $s\in S$. If each $X_s$ is a reflexive Banach space, then $X_{-\infty}$ is a Hausdorff LCS and $X^* \eqdef \{X^*_{-s}\}_{s \in -S}$ is a Banach scale on $-S = \{-s: s \in S\}$ with bonding maps $(\iota^X_{-r,-s})^*: X^*_{-s} \hookrightarrow X^*_{-r}$, $s>r \in -S$, and limit $(X_{-\infty}^*, \{(\iota^X_{-s, -\infty})^*\}_{s\in -S})$. \label{prop:dual_scale}
	\end{enumerate}	
\end{proposition}
\begin{proof}
	The crux of \ref{prop:shifted_scale} is that a cone to $X$ induces a cone to $X^\tau$ and vice\hyp{}versa by shifting indices by $\tau$. Similarly, \ref{prop:direct_sum_scale} is proven by adjoining the universal cones $X_\infty$ and $Y_\infty$ to a cone $X_\infty \times Y_\infty$ to $X\times Y$ and proving universality by noting that if $(Z, \{f_s\})$ is a cone to $X\times Y$, then $(Z, \{\pr^X_s \, f_s\})$ and $(Z, \{\pr^Y_s \, f_s\})$ are cones to $X$ and $Y$, respectively, where $\pr^X_s: X_s \times Y_s \mapsto X_s$ and $\pr^Y_s: X_s \times Y_s \mapsto Y_s$ are the canonical projections. Also, products of compact operators are compact.
	
	As to \ref{prop:dual_scale}, first note that $(\iota^X_{-r,-s})^*$ is compact for $s>r\in -S$ since the adjoint of a compact operator is compact \cite[Theorem~4.19]{rudin91}. Secondly, the map $(\iota^X_{-r,-s})^*$ is injective since $\iota^X_{-r,-s}$ has dense range \cite[Theorem~4.12]{rudin91}. Thirdly, since $X_{-s}$ and $X_{-r}$ are reflexive, the double adjoint $(\iota^X_{-r,-s})^{**} = \iota^X_{-r,-s}: X_{-r}^{**} \iso X_{-r} \mapsto X_{-s} \iso X_{-s}^{**}$ is injective, which again by \cite[Theorem~4.12]{rudin91} implies that $(\iota^X_{-r,-s})^{*}$ has dense range. Consequently, by \cite[Proposition~2]{dubinsky72}, the limit $(X^*)_\infty$ of $X^*$ is dense in each $X_{-s}^*$, $s\in -S$ and by \cite[Proposition~3]{dubinsky72}, the colimit $X_{-\infty}$ of $X$ is Hausdorff. The computation of the limit $(X^*)_\infty \iso X_{-\infty}^*$ is a standard result of category theory, and a thorough derivation can be found in \cite{crespo17,leinster14}.
\end{proof}
\begin{remark}
	Similar remarks hold as for handling systems of subsets:
	\begin{enumerate}
		\item Although the shifted scale $X^\tau$ is defined on $S-\tau$, in the sequel of this document, we will restrict this type of scales to $S \cap (S-\tau)$.
		\item For the product scale, if $\{X_s\}_{s \in S}$ and $\{Y_s\}_{s \in S'}$ are defined on different subsets $S,S'\subset\setZ$, one restricts both scales to the overlap $S\cap S'$ before applying the product construction.
	\end{enumerate}
\end{remark}

Our next goal is to study maps between scales and their systems of subsets. In the roughest case, a map between scales can be a non\hyp{}linear map between each layer. Eventually, we can add more structure to the map and require each layer to be linear or even continuous. This gives rise to the notion of morphisms and isomorphisms of scales.
\begin{definition} \thlabel{def:maps_scales}
	Let $X = \{X_s\}_{s \in S}$, $Y = \{Y_s\}_{s \in S}$ and $Z = \{Z_s\}_{s\in S}$ be Banach scales on $S\subset\setZ$, and let $A\subset X$, $B\subset Y$ and $C\subset Z$ be systems of subsets.
	\begin{enumerate}
		\item A map between the two systems of subsets $A$ and $B$ is a family of functions $f \eqdef \{f_s: A_s \mapsto B_s\}_{s\in S}$ which satisfies the compatibility requirement $f_r\restr{A_s} = f_s$ for all $s>r\in S$ (as maps from $A_s$ to $B_r$). Explicitly, this means that the diagram
\[ \begin{tikzcd}
	A_s \ar{r}{f_s} \ar[hook]{d}{\iota^X_{sr}}	& B_s \ar[hook]{d}{\iota^Y_{sr}} \\
	A_r \ar{r}{f_r}								& B_r
\end{tikzcd} \]
is required to commute for all $s>r$. When the systems of subsets are clear from context, $f$ is simply said to be a map of scales or a scale map. A map between the (full) Banach scales $X$ and $Y$ is simply defined as a map between the trivial systems of subsets $A=X$ and $B=Y$.
	\item Composition of maps of scales is defined layer\hyp{}wise: if $f: A \mapsto B$ and $g: B \mapsto C$ are maps of scales, then $g\circ f: A \mapsto C$ defined by $(g\circ f)_s = g_s \circ f_s: A_s \mapsto C_s$, $s\in S$, is a map of scales. We define the identity map $\id_A: A \mapsto A,\, (\id_A)_s = \id_{A_s}$.
	\item A map of scales $f: A \mapsto B$ is called injective, surjective or bijective if each $f_s: A_s \mapsto B_s$ is injective, surjective or bijective, respectively. A bijective map of scales $f: A \mapsto B$ defines an inverse map $f^{-1}: B \mapsto A,\, (f^{-1})_s = {f_s}^{-1}$ with $f^{-1}\circ f = \id_A$ and $f\circ f^{-1} = \id_B$.
	\item A map of scales $L: A \mapsto B$ is called linear if $A_s \subset X_s$ and $B_s \subset Y_s$ are linear subspaces and all $L_s: A_s \mapsto B_s$ are linear. 
	\item A map of scales $f: A \mapsto B$ is called continuous or $\scale^0$ if $f_s: A_s \mapsto B_s$ is continuous for all $s \in S$, where $A_s$ and $B_s$ inherit the topologies of $X_s$ and $Y_s$, respectively. It is an isometry if each $f_s$ is an isometry.
	\item \label{def:scale_morphisms} A morphism between $X$ and $Y$, also called an $\scale$\hyp{}operator, is an $\scale^0$ linear map $T: X \mapsto Y$. An isomorphism between $X$ and $Y$, or an $\scale$\hyp{}isomorphism, is a morphism $T: X \mapsto Y$ for which $T_s: X_s \isomapsto Y_s$ is an isomorphism of vector spaces for all $s\in S$. It is an isometric isomorphism if, in addition, it is an isometry.
	\end{enumerate}
\end{definition}
\begin{remark} \thlabel{rmk:maps_scales}
	\begin{enumerate}
		\item If $T: X \isomapsto Y$ is an isomorphism of scales as defined above, the open mapping theorem implies that $T^{-1}: Y \mapsto X$, defined by $(T^{-1})_s = T_s^{-1}$, is also a morphism of scales.
		\item A morphism $T: X \mapsto Y$ defines a unique continuous linear map, the limit map $T_\infty \eqdef \lim_{s\in S} T_s: X_\infty \mapsto Y_\infty$, with the property $\iota_{\infty s}^Y \, T_\infty = T_s \,\iota_{\infty s}^X$ for all $s\in S$. Similarly, it defines the colimit map $T_{-\infty} \eqdef \colim_{s\in S} T_s: X_{-\infty} \mapsto Y_{-\infty}$ uniquely by the property $T_{-\infty} \, \iota_{s,-\infty}^X = \iota_{s,-\infty}^Y \, T_s$ for all $s$. These are isomorphisms of topological vector spaces if $T$ is an $\scale$\hyp{}isomorphism.
		\item If $X$ and $Y$ admit dual scales, a morphism $T: X \mapsto Y$ defines an adjoint morphism $T^* \eqdef \{T_{-s}^{~*}\}_{s\in -S}: Y^* \mapsto X^*$.
		\item \label{rmk:inf_in_scale} If $S$ is bounded below in $\setZ$, with $r \eqdef \min S$, any subset $A_r \subset X_r$ defines a system of subsets by putting $A_s \eqdef A_r \cap X_s$, $s\geq r$, and this system is open if $A_r \subset X_r$ is open. Similar considerations hold for $B_r\subset Y_r$. A function $f_r: A_r \mapsto B_r$ with $f_r(A_s) \subset B_s$ for all $s \geq r$ then defines a unique map of scales $f: A \mapsto B$ which is linear if $f$ is linear, and all maps of scales arise in this way. Note that although $f_r$ injective implies $f$ injective, the same cannot be said about surjectivity. A counterexample is the inclusion $I: X\restr{2\setNz} \mapsto X$, $I_k = \iota^X_{2k,k}: X_{2k} \hookrightarrow X_k$, $k\in\setNz$, for any proper scale $X$ on $\setNz$.
		\item If $X$ and $Y$ are Banach scales on different index sets, say $S$ and $S'$, respectively, we restrict $A\subset X$ and $B\subset Y$ to $S\cap S'$, and define maps of scales $A \mapsto B$ simply as being maps of scales $A\restr{S\cap S'} \mapsto B\restr{S\cap S'}$.
		\item \label{rmk:induced_scale} {\it (induced scales)} If $X = \{X_s\}_{s \in S}$ is now simply a collection of Banach spaces without bonding maps {\it a priori}, $Y = \{Y_s\}_{s \in S}$ is a Banach scale, and we are given continuous linear isomorphisms $\Psi_s: X_s \isomapsto Y_s$ for $s\in S$, then there is a unique structure of a Banach scale on $X$ such that $\Psi = \{\Psi_s\}_{s\in S}$ is an isomorphism of scales. This structure is obtained by pulling back the bonding maps of $Y$, and the required properties for a Banach scale on $X$ are directly derived from the corresponding properties on $Y$. 
	\end{enumerate}
\end{remark}

We conclude this subsection with some basic definitions of linear symplectic geometry for Banach spaces and scales. A (strong) symplectic form on a real Banach space $X$ is a continuous skew\hyp{}symmetric bilinear form $\omega: X \times X \mapsto \setR$ such that the induced map $\iota_\omega: X \mapsto X^*, \, w \mapsto \omega(\cdot, w)$ is an isomorphism of locally convex spaces. The pair $(X,\omega)$ is then called a symplectic Banach space. A symplectic Banach space is always reflexive, since $-(\iota_\omega^{-1})^* \circ \iota_\omega: X \isomapsto X^{**}$ is the canonical injection.

Concerning scales, symplectic structures are defined on Banach scales over the reals and on an index set $S\subset\setZ$ with $S=-S$. A symplectic structure on such a scale $X = \{X_s\}_{s\in S}$ is a skew\hyp{}symmetric collection $\omega = \{\omega_s\}_{s\in S}$ of continuous and bilinear forms $\omega_s: X_s \times X_{-s} \mapsto \setR$, $s\in S$, which induce an isomorphism of scales $\iota_\omega: X \isomapsto X^*, \, X_{s} \owns w \mapsto \omega_{-s}(\cdot, w) \in X_{-s}^*$. In this context, skew\hyp{}symmetry means $\omega_s(v,w) = -\omega_{-s}(w,v)$ for all $v\in X_s$, $w\in X_{-s}$ and $s\in S$. Due to continuity, it is enough to check this condition for smooth vectors $v,w \in X_\infty$. Also, note that the existence of the dual scale $X^*$ is an immediate consequence of the individual isomorphisms in $\iota_\omega$ since, similarly as for the single Banach space, $-((\iota_{\omega})_{-s}^{~-1})^* \circ (\iota_{\omega})_s: X_s \isomapsto X_s^{**}$ is the canonical injection. The pair $(X, \omega)$ is called a symplectic Banach scale.

\begin{remark} \thlabel{rmk:symplectic_scale_extension}
	By using the intrinsic identification of a symplectic Banach scale with its dual, we can extend single\hyp{}sided Banach scales which are isomorphic to a given symplectic scale. Indeed, let $(X,\omega)$ be a symplectic Banach scale on $S = -S \subset\setZ$, and let $Y$ be a Banach scale on $S_{\geq 0} \eqdef S \cap \setNz$ with isomorphism of scales $\Psi: Y \isomapsto X\restr{S_{\geq 0}}$. Let also $S_{> 0} \eqdef S \cap \setN$. The adjoint of $\Psi$ and $\iota_\omega$ induce an isomorphism
	\begin{equation}
		(\Psi^{-1})^*: (Y\restr{S_{>0}})^* \isomapsto (X\restr{S_{>0}})^* \stackrel[\sim]{\iota_\omega}{\mapsfrom} X\restr{-S_{>0}} \,.
	\end{equation}
	Define $Y_s \eqdef Y_{-s}^*$ and $\Psi_s \eqdef (\Psi_{-s}^{~-1})^*: Y_s \isomapsto X_s$ for $s\in -S_{>0}$. As in \thref{rmk:maps_scales}\ref{rmk:induced_scale}, we obtain bonding maps for the extended collection $Y = \{Y_s\}_{s\in S}$ making it into a Banach scale with extended isomorphism $\Psi = \{\Psi_s\}_{s\in S}: Y \isomapsto X$. In particular, by setting $Y = X\restr{S_{\geq 0}}$ and $\Psi = \id_X$, we see that a symplectic Banach scale $(X,\omega)$ is completely determined by its one\hyp{}sided structure $X\restr{S_{\geq 0}}$.
\end{remark}

If $(X, \omega)$ and $(Y, \eta)$ are symplectic Banach scales on $S=-S \subset \setZ$ and $S'\subset S$ is a non\hyp{}empty subset, a morphism of scales $T: X\restr{S'} \mapsto Y\restr{S'}$ always induces a symplectic adjoint
\begin{equation}
	T^{\omega,\eta}: Y\restr{-S'} \stackrel[\sim]{\iota_\eta}{\mapsto} (Y\restr{S'})^* \stackrel{T^*}{\longmapsto} (X\restr{S'})^* \stackrel[\sim]{\iota_\omega}{\mapsfrom} X\restr{-S'}
\end{equation}
uniquely defined by the relation $\eta(Tv, w) = \omega(v, T^{\omega,\eta} w)$ for all $v \in X_{s}$, $w \in Y_{-s}$ and $s \in S'$. If $S' \cap (-S') \neq \emptyset$, such a morphism $T$ is called symplectic if $\eta(Tv, Tw) = \omega(v, w)$ for all $v \in X_s$, $w \in X_{-s}$ and $s\in S'\cap (-S')$. Again due to continuity, it is enough to check this for $v,w\in X_\infty \subset (X\restr{S'})_\infty$. It is easy to see that a morphism $T$ is symplectic if and only if
\begin{equation}
	T^{\omega,\eta} \circ T = \id_{X\restr{S'\cap(-S')}} \,.
\end{equation}
Note that since $(T^{\omega,\eta})^{\eta,\omega} = T$, $T^{\omega,\eta}$ is symplectic if and only if $T \circ T^{\omega,\eta} = \id_{Y\restr{S'\cap(-S')}}$. Also, both $T^{\omega,\eta} \circ T = \id_X$ and $T \circ T^{\omega,\eta} = \id_Y$ hold on $S'\cap (-S')$ if and only if $T$ is an isomorphism of scales on $S'\cap(-S')$ which is symplectic. In that case, $T^{-1} = T^{\omega,\eta}$ on $S'\cap(-S')$ and we call $T$ a linear symplectomorphism of scales.

\subsection{Hilbert Scales} \label{sec:hilbert_scales}
A Hilbert scale is a Banach scale $\{X_s\}_{s\in S}$, where each $X_s$ is required to be a Hilbert space. This important special case of Banach scales arises naturally in symplectic geometry: starting from a complex separable infinite\hyp{}dimensional Hilbert space $X$, we can define a linear symplectic form $\omega: X \times X \mapsto \setR$, a scale structure $\scf{X} = \{X_s\}_{s\in\setZ}$ with center $X_0 \iso X$, and eventually a symplectic structure on $\scf{X}$ derived from $\omega$. To differentiate between Hilbert spaces and scales, and also to avoid ambiguity, we underline Hilbert scales and scale maps in this section.

We start with a prototypical example of a Hilbert scale, $\scf{l^2_\nu}$, which characterizes all Hilbert scales we will be dealing with. This is a scale on $\setZ$ with center $l^2 = \{x \in \setF^\setZ: \sum_{n\in\setZ} |x_n|^2 < \infty \}$, and where the remaining spaces are weighted according to a positive sequence $\nu$. On this scale, we can identify spaces on the one side of the scale with the duals of their symmetric counterparts. Using this prototype we build, for every separable infinite\hyp{}dimensional Hilbert space $X$, a scale $\scf{X}$ on $\setZ$ by pulling back the scale structure of $\scf{l^2_\nu}$ using the isometric isomorphism arising from a Hilbert basis $\{\phi_k\}_{k\in\setZ}$. In principle, it is also possible to have these scales indexed on $\setR$ (it is not difficult to extend the theory this index set), but that will be rarely needed in this document.

\begin{proposition} \thlabel{prop:l2s_scale}
	Let $\nu \in \setRppos^\setZ$ be a sequence with $\nu_n \to \infty$ as $|n|\to \infty$, define
	\begin{equation} \label{eq:l2s_space}
		l^{2,s} \eqdef l_\nu^{2,s} \eqdef \left\{x \in \setF^\setZ: \sum_{n\in\setZ} |x_n|^2 \nu_n^{2s} < \infty \right\}
	\end{equation}
	for $s\in\setZ$, and endow this vector space with the (real or hermitian) inner product
	\begin{equation}
		\langle x, y\rangle_s \eqdef \sum_{n\in\setZ} x_n \conj{y_n} \, \nu_n^{2s} \,.
	\end{equation}
	Then each $l^{2,s}$ is a Hilbert space and $\scf{l^2} \eqdef \scf{l_\nu^2} \eqdef \{l_\nu^{2,s}\}_{s\in\setZ}$ is a Hilbert scale on $\setZ$ with limit $l^{2,\infty} \eqdef \cap_{s\in\setZ} l^{2,s}$ and colimit $l^{2,-\infty} \eqdef \cup_{s\in\setZ} l^{2,s}$. Furthermore, the collection $F \eqdef \Sp\{\delta_k: k\in\setZ\} \subset l^{2,\infty}$ of finite linear combinations of the standard basis $\delta_k(n) = \delta_{kn}$ is dense in each $l^{2,s}$, $s\in\setZ\cup\{\infty\}$.
\end{proposition}
\begin{proof}
	Each $l^{2,s}$, $s\in\setZ$, is a weighted $l^2$ space, and consequently a Hilbert space. It is also clear that $l^{2,s} \subset l^{2,r}$ for $s>r \in \setZ$, hence $\scf{l^2} = \{l^{2,s}\}_{s\in\setZ}$ is a descending filtration with mentioned limit and colimit. Let $s\in\setZ\cup\{\infty\}$ and define, for $k\in \setN$, the projection $p_s^k: l^{2,s} \mapsto F$, $x \mapsto x \cdot \ind_{\{-k+1,-k+2,\ldots,k-1\}}$, where the product $\cdot$ is pointwise and $\ind_{(\cdot)} \in \{0,1\}^\setZ$ is the indicator function. Then $F$ is dense in $l^{2,s}$: if $s<\infty$, for $x\in l^{2,s}$, $\|x - p_s^k(x)\|_s \to 0$ as $k\to\infty$, and if $s=\infty$, $\|x - p_s^k(x)\|_{s'} \to 0$ for all $s'\in\setZ$. This implies, in particular, that $l^{2,\infty}$ is dense in $l^{2,s}$ for each $s\in\setZ$. Furthermore, for $\setZ\owns s>r$, let $\iota_{sr}: l^{2,s} \hookrightarrow l^{2,r}$ and $\iota_{\infty r}: l^{2,\infty} \hookrightarrow l^{2,r}$ be the inclusions and define $p^k_{sr} \eqdef \iota_{\infty r}\,p_s^k: l^{2,s} \mapsto l^{2,r}$, $k\in\setN$. Fix $\epsilon>0$ and let $N\in\setN$ be such that $\nu_n^{2(r-s)} \leq \epsilon$ for all $|n|\geq N$. For $x\in l^{2,s}$ with $\|x\|_s \leq 1$, we have that whenever $k\geq N$,
	\begin{equation}
		\|(\iota_{sr}-p^k_{sr})(x)\|_r^2 = \sum_{|n|\geq k} |x_n|^2 \nu_n^{2s} \nu_n^{2(r-s)} \leq \epsilon \|x\|_s^2 \leq \epsilon \,.
	\end{equation}
	From this, we conclude that $\|\iota_{sr}-p^k_{sr}\|_{\B(l^{2,s},l^{2,r})} \to 0$.
Seen that $p^k_{sr}$ is a finite rank operator for each $k\in\setN$ and that the compact operators are a closed subset of $\B(l^{2,s},l^{2,r})$, we conclude that $\iota_{sr}$ is compact.
\end{proof}

\begin{lemma} \thlabel{lem:l2s_isodual}
	Let $\nu \in \setRppos^\setZ$ and $\scf{l_\nu^2}$ be as in \thref{prop:l2s_scale}. Define complex conjugation $\conj{(\cdot)}: l^{2,s} \mapsto l^{2,s}$ pointwise\footnote{This is a conjugate\hyp{}linear isomorphism which squares to the identity when $\setF = \setC$ and simply the identity map when $\setF = \setR$.}. Then
	\begin{enumerate}
		\item \label{lem:l2s_cont} For $s\in\setZ$, we have an ($\setF$-)bilinear continuous pairing $l^{2,s} \times l^{2,-s} \mapsto \setF,\, (x,y) \mapsto \langle x, \conj{y} \rangle_0 = \sum_{n\in\setZ} x_n y_n$.
		\item \label{lem:l2s_iso_expression} The induced map $l^{2,-s} \mapsto (l^{2,s})^*$, $y \mapsto \langle \cdot, \conj{y} \rangle_0$ is an isometric isomorphism with inverse $D \mapsto \{D(\delta_n)\}_{n\in\setZ}$ for each $s\in\setZ$. It induces an isometric isomorphism of scales $\scf{l^2} \isomapsto (\scf{l^2})^*$.
		\item \label{lem:l2-inf_iso} We have an isomorphism of vector spaces $l^{2,-\infty} \isomapsto (l^{2,\infty})^*$ given by the same formula (and same inverse).
	\end{enumerate}
\end{lemma}
\begin{proof}
	For the first statement, note that for $x\in l^{2,s}$ and $y\in l^{2,-s}$, Hölder's inequality gives
	\begin{equation}
		|\langle x, \conj{y} \rangle_0| \leq \sum_{n\in\setZ} \big(|x_n| \nu_n^{s}\big) \big(|y_n| \nu_n^{-s} \big) \leq \|x\|_{s} \|y\|_{-s} \,,
	\end{equation}
	so that $\langle \cdot, \conj{\,\cdot\,} \rangle_0$ is defined and continuous (it is clearly bilinear). As far as the second statement is concerned, by the above, with $y\in l^{2,-s}$, $\|\langle \cdot, \conj{y} \rangle_{0} \| \leq \|y\|_{-s}$. But also with $x_n \eqdef \conj{y_n} \nu_n^{-2s}$, we have $x=\{x_n\}_{n\in\setZ}\in l^{2,s}$, $\|x\|_{s} = \|y\|_{-s}$ and $|\langle x, \conj{y} \rangle_0| = \|y\|_{-s}^2$. We thus conclude that $y \mapsto \langle \cdot, \conj{y} \rangle_0$ is an isometry. Now let $D = \langle \cdot, x \rangle_{s} \in (l^{2,s})^*$ for $x\in l^{2,s}$. Then $D(\delta_n) = \conj{x_n} \nu_n^{2s}$ and $\{\conj{x_n}\nu_n^{2s}\}_{n\in\setZ}\in l^{2,-s}$, hence we have a well\hyp{}defined candidate inverse map. The fact that it is indeed an inverse comes from the fact that $\Sp\{\delta_k: k\in\setZ\} \subset l^{2,s}$ is dense (\thref{prop:l2s_scale}). Finally, since Hilbert spaces are reflexive, $(\scf{l^2})^*$ exists and the scale isomorphism is a direct consequence of the former isomorphisms.
	
	To prove the last statement, we first note that the direct map is well defined, since $(l^{2,s})^* \subset (l^{2,\infty})^*$ for all $s\in\setZ$. For the inverse map, we use the fact that since $l^{2,\infty}$ is a limit of the Banach spaces $l^{2,s}$, any $D\in (l^{2,\infty})^*$ factors through some $l^{2,s}$ \cite[Theorem~5.1.1]{garrett13}. Choose $D\in (l^{2,\infty})^*$ thus arbitrarily and let $s\in\setZ$ and $D_s \in (l^{2,s})^*$ be such that $D = D_s \, \iota_{\infty s}$. Then $D(\delta_n) = D_s(\delta_n) = y_n$, where $y \in l^{2,-s} \subset l^{2,-\infty}$ is the result of the (inverse) isomorphism in \ref{lem:l2s_iso_expression} applied to $D_s$. It follows that the candidate inverse is well\hyp{}defined, and again we invoke \thref{prop:l2s_scale} to complete the proof. 
\end{proof}

Note that even if we choose $\setF = \setC$ in \thref{prop:l2s_scale}, the complex Hilbert spaces $l_\setC^{2,s} = l^{2,s}$ can still be regarded as a \emph{real} Hilbert spaces $(l_\setC^{2,s})_\setR$ by restricting scalar multiplication to the reals and using the inner product $\Re\{ \langle \cdot, \cdot \rangle_s \}$. It is not difficult to see that $\{(l_\setC^{2,s})_\setR\}_{s\in\setZ}$ is still a Hilbert scale (over $\setR$). The following corollary, which is needed to handle real symplectic forms on complex Hilbert spaces, extends \thref{lem:l2s_isodual} to this scale.
\begin{corollary} \thlabel{cor:l2s_complexiso}
	In the setting of \thref{lem:l2s_isodual} with $\setF = \setC$, let $l_\setC^{2,s} = l_{\nu,\setC}^{2,s}$ be Eq.~\eqref{eq:l2s_space}. Define $\tilde\nu \in \setRppos^\setZ$ by $\tilde\nu_{2n} = \tilde\nu_{2n+1} \eqdef \nu_n$, $n\in\setZ$, and let $l_\setR^{2,s} = l_{\tilde \nu,\setR}^{2,s}$ be Eq.~\eqref{eq:l2s_space} with $\setF = \setR$ instead of $\setC$ and $\tilde\nu$ instead of $\nu$. Then we have a continuous pairing
	\begin{equation}
		(l_\setC^{2,s})_\setR \times (l_\setC^{2,-s})_\setR \mapsto \setR,\, (x,y) \mapsto \Re\{\langle x, y \rangle_0\} = \Re\left\{\sum_{n\in\setZ} x_n \conj{y_n}\right\}
	\end{equation}
	which induces an isometric isomorphism.
\end{corollary}
\begin{proof}
	By choosing the orthonormal basis of $(l_\setC^2)_\setR$ given by $\tilde\delta_{2k} = \delta_k$, $\tilde\delta_{2k+1} = \i\delta_k$, $k\in\setZ$, we obtain ($\setR$\hyp{}linear) isometric isomorphisms $(l_{\setC}^{2,s})_\setR \isomapsto l_{\setR}^{2,s},\, x \mapsto (\ldots, x^\txt{R}_{-1}, x^\txt{I}_{-1}, x^\txt{R}_{0}, x^\txt{I}_{0},\allowbreak x^\txt{R}_{1}, x^\txt{I}_{1}, \ldots)$, $s\in\setZ$, where $x_n = x_n^\txt{R} + \i x_n^\txt{I}$ is the real\hyp{}imaginary decomposition. Composing this map with the pairing of \thref{lem:l2s_isodual}\ref{lem:l2s_cont} delivers the desired pairing.
\end{proof}

Now let $X$ be a separable infinite\hyp{}dimensional Hilbert space with orthonormal basis $\{\phi_k\}_{k\in\setZ} \subset X$. This basis induces an isometric isomorphism $\Phi: X \isomapsto l^2, v \mapsto \{\langle v, \phi_k \rangle\}_{k\in\setZ}$. Let also $\nu \in \setRppos^\setZ$ be as in the above discussion. From $X$, we construct a Hilbert scale $\scf{X} = \{X_s\}_{s\in\setZ}$ with center $X_0 \iso X$ as follows. Restrict $\Phi$ to an isomorphism of vector spaces $\Phi_\infty \eqdef \Phi\restr{X_\infty}: X_\infty \isomapsto l^{2,\infty}$, with $X_\infty \eqdef \Phi^{-1}(l^{2,\infty}) \subset X$, and pull the limit topology of $l^{2,\infty}$ back to $X_\infty$. Subsequently, define the isomorphism $\Phi_{-\infty}: (X_\infty)^* \iso (l^{2,\infty})^* \isomapsto l^{2,-\infty}, D \mapsto \{D(\phi_n)\}_{n\in\setZ}$ using the map of \thref{lem:l2s_isodual}\ref{lem:l2-inf_iso}. Then, in a similar fashion, we can restrict $\Phi_{-\infty}$ to isomorphisms $\Phi_s \eqdef \Phi_{-\infty}\restr{X_s}: X_s \isomapsto l^{2,s}$, with $X_s \eqdef \Phi_{-\infty}^{-1}(l^{2,s})$, and pull the inner product of $l^{2,s}$ back to $X_s$, $s\in\setZ$. By construction, $\scf{X} \eqdef \{X_s\}_{s\in\setZ}$ is a Hilbert scale and $\scf{\Phi} \eqdef \{\Phi_s\}_{s\in\setZ}: \scf{X} \isomapsto \scf{l^2}$ is an isometric isomorphism of scales. Consequently, all properties of $\scf{l^2}$ carry directly over to $\scf{X}$. The following proposition reveals some of these properties.

\begin{proposition} \thlabel{prop:separable_hilbert_scale}
	Let $X$ be a separable infinite\hyp{}dimensional Hilbert space with orthonormal basis $\{\phi_k\}_{k\in\setZ}$ and corresponding isometric isomorphism $\Phi: X \isomapsto l^2$. Furthermore, let $\nu \in \setRppos^\setZ$ as in \thref{prop:l2s_scale}, $\scf{X} = \{X_s\}_{s\in\setZ}$ be the corresponding induced Hilbert scale with isometric isomorphism of scales $\scf{\Phi}: \scf{X} \isomapsto \scf{l^2}$. To ease notation, pull complex conjugation back to $X_s$, {\it i.e.}, define $\conj{(\cdot)}: X_s \mapsto X_s,\, v \mapsto \Phi_s^{-1}(\conj{\Phi_s v})$, $s\in\setZ$. Then the following holds:
	\begin{enumerate}
		\item For $s\geq 0$, we can make the identification
			\begin{equation} \label{eq:hs_sgeq0}
				X_s \iso \Phi^{-1}(l^{2,s}) = \{v\in X: \|\Phi(v)\|_s < \infty\} \,.
			\end{equation}
		\item The limit of $\scf{X}$ is $\cap X_s \iso X_\infty$ as topological vector spaces ($D \isomapsto \sum_{n\in\setZ} D(\phi_n) \phi_n$) and the colimit is $X_{-\infty} \eqdef \cup X_s = X_{\infty}^*$.
		\item \label{prop:separable_hilbert_iso} We have a continuous pairing $X_s \times X_{-s} \mapsto \setF,\, (v,w) \mapsto \langle v, \conj{w} \rangle_0$. This map induces an isometric isomorphism $X_{-s} \isomapsto X_s^*$ for each $s\in\setZ$, and hence an isometric isomorphism of scales $\scf{X} \isomapsto \scf{X}^*,\, w \mapsto \langle \cdot, \conj{w} \rangle_0$.
	\end{enumerate}
\end{proposition}
\begin{proof}
	The proof boils down to composing obvious maps. For $s\geq 0$, $\Phi_s: X_s \isomapsto l^{2,s}$ and $\Phi\restr{\Phi^{-1}(l^{2,s})}: \Phi^{-1}(l^{2,s}) \isomapsto l^{2,s}$ are vector space isomorphisms which are isometries by construction, provided we pull the inner product of $l^{2,s}$ back to the corresponding spaces. The isomorphism for the limit is obtained using the limit map $\lim_{s\in\setZ} \Phi_s: \cap_{s\in\setZ} X_s \isomapsto l^{2,\infty}$ and $\Phi_\infty: X_\infty \isomapsto l^{2,\infty}$. For the remaining statements, one can use \thref{lem:l2s_isodual}.
\end{proof}

\begin{example} \thlabel{ex:sobolev_scale}
	For the separable complex Hilbert space
	\begin{equation}
		L^2(S^1, \setC) = \Big\{u: S^1 \mapsto \setC: u~\text{measurable and}~\int_{S^1} |u|^2 < \infty \Big\}
	\end{equation}
	with Fourier basis $\left\{x \mapsto \frac{\e^{\i kx}}{\sqrt{2\pi}}: k \in\setZ\right\}$, inner product
	\begin{equation}
		\langle u, v \rangle = \int_{S^1} u(x) \conj{v(x)} \,\d x \,,
	\end{equation}
	and sequence $\nu_n \eqdef (1+n^2)^{1/2}$, $n \in\setZ$, the induced Hilbert scale is the scale of Levi\hyp{}Sobolev spaces \cite{garrett13}
	\begin{equation}
		W^{s,2}(S^1, \setC) = \Big\{D \in C^\infty(S^1, \setC)^*: \sum_{n\in\setZ} \big| D\big(x \mapsto \tfrac{\e^{\i nx}}{\sqrt{2\pi}}\big)\big|^2 (1+n^2)^s < \infty \Big\}
	\end{equation}
	for $s\in\setZ$, with the smooth functions $C^\infty(S^1, \setC) = \lim_{k\in\setN} C^k(S^1, \setC)$ as limit and the distributions $C^\infty(S^1, \setC)^*$ as colimit. For $s\geq 0$, these spaces are simply
	\begin{equation}
		W^{s,2}(S^1, \setC) \iso \Big\{u \in L^2(S^1, \setC): \sum_{n\in\setZ} |{\hat u}_n|^2 (1+n^2)^s < \infty \Big\} \,,
	\end{equation}
	where
	\begin{equation}
		{\hat u}_n = \frac{1}{\sqrt{2\pi}} \int_{S^1} u(x) \e^{-\i nx} \,\d x
	\end{equation}
	is the $n\th$ Fourier coefficient of $u$ and where we identify $L^2$ functions with the subspace $W^{0,2} \subset (C^\infty)^*$ of distributions by sending $u \in L^2$ to $(\varphi \mapsto \sum_{n\in\setZ} {\hat u}_n {\hat \varphi}_n = \int_{S^1} u(x) \varphi(-x) \,\d x) \in W^{0,2}$.
\end{example}

\begin{remark} \thlabel{rmk:hilbert_scale_extension}
	Similarly to \thref{rmk:symplectic_scale_extension}, a single\hyp{}sided Hilbert scale which is isomorphic to $\scf{X}$ extends to a double\hyp{}sided scale, but now using the isometric isomorphism of \thref{prop:separable_hilbert_scale}\ref{prop:separable_hilbert_iso} induced by the inner product of $X$ instead of a symplectic structure. Clearly, the extended isomorphism is isometric if and only if the original isomorphism is so. Again, $\scf{X}$ is completely determined by its one\hyp{}sided structure $\scf{X}\restr\setNz$.
\end{remark}

Regarding linear symplectic geometry on Hilbert spaces, if $(X, \langle \cdot, \cdot \rangle)$ is a complex Hilbert space, it can be given the structure of a real Hilbert space $X_\setR = X$ with inner product $\langle \cdot, \cdot \rangle_\setR \eqdef \Re\{ \langle \cdot, \cdot \rangle \}$ in a way similar to the discussion prior to \thref{cor:l2s_complexiso}. With this structure, $\omega = -\Im\{ \langle \cdot, \cdot \rangle \}$ is a symplectic form on $X_\setR$, denoted the standard symplectic form. It is compatible with the inner product $\langle \cdot, \cdot \rangle_\setR$ and the complex structure $\i: X \mapsto X, \, v \mapsto \i v$, in the sense that $\omega = \langle \i \cdot, \cdot \rangle_\setR$.

In the same way, if $X$ is a separable infinite\hyp{}dimensional complex Hilbert space and $\scf{X} = \{X_s\}_{s\in\setZ}$ is the construction of \thref{prop:separable_hilbert_scale}, we can regard each $X_s$ as a real Hilbert space $(X_s)_\setR$. Thereby, we obtain isometric ($\setR$\hyp{}linear) isomorphisms $\Phi_s: (X_s)_\setR \isomapsto (l_\setC^{2,s})_\setR$, $s\in\setZ$, which give rise to a scale structure (over $\setR$) on $\{(X_s)_\setR\}_{s\in\setZ}$. By abuse of notation, we also denote this restriction\hyp{}of\hyp{}scalars scale by $\scf{X}$. By \thref{cor:l2s_complexiso}, the map $\langle \cdot, \cdot\rangle_{0,\setR} \eqdef \Re\{\langle \cdot, \cdot\rangle_0\}: (X_s)_\setR \times (X_{-s})_\setR \mapsto \setR, \, (v,w) \mapsto \Re\{\langle v, w \rangle_0\}$ induces an isomorphism of scales, hence $\omega = \langle \i \cdot, \cdot\rangle_{0,\setR} = -\Im\{ \langle \cdot, \cdot \rangle_0 \}$ furnishes $\scf{X}$ with a symplectic structure, denoted the standard symplectic structure on the scale $\scf{X}$ induced by $X$.

\begin{example} \thlabel{ex:l2s_sobolev_sympl}
	For the prototypical case $X = l^2_\setC$, the previous construction endows $\{(l_\setC^{2,s})_\setR\}_{s\in\setZ}$ with the symplectic structure
	\begin{equation} \label{eq:symplectic_l2s}
		l_\setC^{2,s} \times l_\setC^{2,-s} \mapsto \setR, \, (x,y) \mapsto -\Im\left\{ \sum_{n\in\setZ} x_n \conj{y_n} \right\}
	\end{equation}
	which pulls back for $N\in\setN$ via the embedding $\setC^{2N-1} \hookrightarrow l_\setC^{2,\infty}$
	\begin{equation}
		x = (x_{-N+1}, x_{-N+2}, \ldots, x_{N-1}) \mapsto \begin{cases} x_n & 
		\txt{if}~|n|<N \\ 0 & \txt{otherwise} \end{cases}
	\end{equation}
	to the standard symplectic form $\frac{\i}{2} \sum_{n=-N+1}^{N-1} \d z_n \wedge \d \conj{z_n}$ on $\setC^{2N-1}$. For \thref{ex:sobolev_scale}, one pre\hyp{}composes \eqref{eq:symplectic_l2s} with the Fourier transform $W^{s,2} \isomapsto l_\setC^{2,s},\, D \mapsto \big\{D\big(x \mapsto \frac{e^{\i nx}}{\sqrt{2\pi}}\big)\big\}_{n\in\setZ}$.
\end{example}

\section{Hofer-Wysocki-Zehnder Scale Smoothness} \label{sec:hwz}
With the necessary working tools for Banach scales in our pockets, our next step is to introduce the notion of $\scale$\hyp{}smoothness developed by Hofer, Wysocki and Zehnder \cite{hofer10,hofer06,hofer08_polyfolds,hofer17}. This notion differs from classical Frèchet smoothness in that it utilizes several layers of regularity to allow for densely\hyp{}defined derivatives and maps which ``loose regularity'' in their infinitesimal form. In a way similar to Banach manifolds, this notion also gives rise to a manifold structure by endowing a topological space with a (maximal) atlas whose transition maps are $\scale$\hyp{}smooth diffeomorphisms. Unless otherwise mentioned, in this section all Banach scales are assumed to be over the real numbers and with index set $\setNz$.

\subsection{Scale Calculus} \label{sec:sc_calculus}
To put the new notion in context, we first recall the classical notion of smoothness between Banach spaces. Let $X$ and $Y$ be real Banach spaces and let $U\subset X$ be open. A function $f: U \mapsto Y$ is said to be Frèchet differentiable if there exists a function $\d f: U \mapsto \B(X,Y)$ with
\begin{equation}
	\lim_{h\to 0} \frac{\|f(x+h) - f(x) - \d f(x)\cdot h\|_{Y}}{\|h\|_{X}} = 0
\end{equation}
for all $x\in U$, where $h\in X$ converges to 0 in $X$. As usual, Frèchet differentiable functions are continuous. Since $\B(X,Y)$ has the structure of a Banach space by itself, one can iterate the concept using $\d f$. Hence, the function $f$ is said to be $C^1$ if it is differentiable and $\d f$ is $C^0$ (continuous) and we define, recursively, $f$ to be $C^{k+1}$ if $\d f$ is $C^k$, $k\geq 1$. The function $f$ is then said to be $C^\infty$ or smooth if it is $C^k$ for all $k\in\setN$. For $X = \setR^m$ and $Y = \setR^n$, $m,n\in\setN$, this notion recovers standard multivariate calculus.

Before proceeding, recall from \thref{rmk:maps_scales}\ref{rmk:inf_in_scale} that if $X$ is a Banach scale on $\setNz$ and $U_0\subset X_0$ is an open subset, $U_0$ induces an open system of subsets $U\subset X$ given by $U_k \eqdef U_0 \cap X_k$, $k\in \setNz$. If $Y$ is another Banach scale on $\setNz$, a map of scales $f: U \mapsto Y$ can be seen as a map $f_0: U_0 \mapsto Y_0$ such that $f_0(U_k) \subset Y_k$ for all $k\in \setNz$. We use this configuration in what follows, always regarding scales, open systems of subsets and maps as their zeroth layers: $X \equiv X_0$, $U \equiv U_0$, $(f: U \mapsto Y) \equiv (f_0: U_0 \mapsto Y_0)$, and so on. We also recall that for $k\in \setNz$, $U^k = \{U_{k+m}\}_{m\in \setNz}$ is the $k$\hyp{}shifted system of subsets. Clearly, it is an open system of subsets of the shifted scale $X^k = \{X_{k+m}\}_{m\in \setNz}$ which is induced by $U_k = (U^k)_{0}$.

\begin{definition} \thlabel{def:sc1}
	Let $X$ and $Y$ be Banach scales (over the reals and on $\setNz$), and $U \subset X$ be open.
	\begin{enumerate}
		\item \label{def:linear_tgbundle} The tangent bundle of $U$ is defined as $\Tg U \eqdef U^1 \times X = \{U_{m+1} \times X_m\}_{m\in \setNz}$. It is an open system of subsets $\Tg U \subset \Tg X = X^1 \times X$ induced by $U_1 \times X_0$.
		\item \label{def:sc1_defn} An $\scale^0$ map $f: U \mapsto Y$ is said to be $\scale^1$ if there exists an $\scale^0$ map $\D f: \Tg U \mapsto Y$ which is linear in the second argument and such that
	\begin{equation} \label{eq:frechet_condition}
		\lim_{h\to 0} \frac{\|f(x+h) - f(x) - \D f(x, h)\|_{Y_0}}{\|h\|_{X_1}} = 0
	\end{equation}
	for all $x\in U_1$, where $h\in X_1$ converges to 0 in $X_1$. We use the notation $\D_x f \eqdef \D f(x,\cdot) \in \B(X_m,Y_m)$ for $x\in U_{m+1}$, $m\in \setNz$.
		\item \label{def:sc1_tgmap} For an $\scale^1$ map $f: U \mapsto Y$, we define its tangent map $\Tg f: \Tg U \mapsto \Tg Y, \, (x,v) \mapsto (f(x), \D f(x, v))$, which is clearly $\scale^0$.
		\item For $k\in \setNz$, $f$ is recursively defined to be $\scale^{k+1}$ if $\Tg f$ is $\scale^k$. In that case, with $\Tg^{k+1} U \eqdef \Tg^k(\Tg U)$ and $\Tg^{k+1} Y \eqdef \Tg^k(\Tg Y)$, the $(k+1)\st$ tangent map of $f$ is defined as $\Tg^{k+1} f \eqdef \Tg^k(\Tg f): \Tg^{k+1} U \mapsto \Tg^{k+1} Y$.		
		\item The map $f$ is said to be $\scale^\infty$ or $\scale$\hyp{}smooth if it is $\scale^k$ for every $k\in \setNz$.
		\item For $U\subset X$ and $V\subset Y$ open and $k\in \setN\cup\{\infty\}$, a bijective scale map $f: U \mapsto V$ is an $\scale^k$\hyp{}diffeomorphism if both $f: U \mapsto V \subset Y$ and $f^{-1}: V \mapsto U \subset X$ are $\scale^k$. For $k=0$, we adopt the terminology $\scale^0$\hyp{}homeomorphism instead.
	\end{enumerate}
\end{definition}
\begin{remark}
	\begin{enumerate}
		\item Re\hyp{}interpreting \eqref{eq:frechet_condition}, we see that in fact we require $f\restr{U_1}: U_1 \mapsto Y_0$ to be Frèchet differentiable with derivative $(\D_x f)\restr{X_1} \in \B(X_1,Y_0)$ for all $x\in U_1$. As mentioned in the following \thref{prop:sc1_alternative}, the definition of an $\scale^1$ map actually implies that $f\restr{U_1}$ is $C^1$.
		\item If we endow each $\B(X_m, Y_m)$ with the compact\hyp{}open topology, we can re\hyp{}interpret the $\scale^0$ condition on $\D f$. We have that $\D f$ is $\scale^0$ if and only if $U_{m+1} \mapsto \B(X_m, Y_m)$, $x \mapsto \D_x f$ is continuous for all $m \in \setNz$.
		\item The tangent map $\Tg f$ is defined even if $f$ is only assumed to be scale\hyp{}differentiable: to define the latter, simply replace ``$\scale^0$ map $\D f$'' by ``scale map $\D f$'' in \thref{def:sc1}\ref{def:sc1_defn}. In that case, $\Tg f$ is $\scale^0$ if and only if $\D f$ is so. Actually, \thref{cor:tf_df_sck} below will prove that $\Tg f$ is $\scale^k$ if and only if $\D f$ is so, $k \in \setNz$.
	\end{enumerate}
\end{remark}

Note that due to the density of $X_1$ in $X_0$, if the $\scale$\hyp{}derivative $\D f$ exists it is unique, and therefore the $\scale^k$ conditions are local, meaning that $f: U \mapsto Y$ is $\scale^k$ if and only if for each $x\in U$, there exists an open neighbourhood $V(x) \subset U$ such that $f\restr{V}: V \mapsto Y$ is $\scale^k$. For notational convenience later on, we define $\scale^k(U,Y) \eqdef \{\scale^k~\txtt{maps}~U \mapsto Y\}$ for $k\in \setNz$.

Hofer et al.~examine several properties of $\scale$\hyp{}smoothness and derive results relating this notion with classical Frèchet differentiability \cite{hofer10,hofer06,hofer08_polyfolds}. Here, we state an equivalent formulation of the $\scale^1$ condition (\thref{def:sc1}\ref{def:sc1_defn}) and present two additional results which make $\scale$\hyp{}smoothness compatible with scale shifting and composition of maps. The first result is of more technical nature, whereas the last two are more essential. Especially the last result is crucial when introducing an $\scale$\hyp{}smooth structure on a topological space so that differentiation behaves well while changing coordinate charts (Section~\ref{sec:sc_manifolds}).
\begin{proposition}[in {\cite[Proposition~2.1]{hofer10}}] \thlabel{prop:sc1_alternative}
	Let $X$ and $Y$ be Banach scales and $U\subset X$ be open. Then an $\scale^0$ map $f: U \mapsto Y$ is $\scale^1$ if and only if for every $m \in \setNz$, the following holds:
	\begin{enumerate}
		\item \label{prop:sc1_alternative_c1} $f\restr{U_{m+1}}: U_{m+1} \mapsto Y_m$ is $C^1$.
		\item \label{prop:sc1_alternative_ext} For $x\in U_{m+1}$, $\d(f\restr{U_{m+1}})(x) \in \B(X_{m+1}, Y_m)$ extends to a continuous operator $\bar{\d}(f\restr{U_{m+1}})(x) \in \B(X_m, Y_m)$. Equivalently, the former linear map is continuous when $X_{m+1}$ inherits the topology of $X_m$.
		\item \label{prop:sc1_alternative_cont} the map $(\D f)_m: U_{m+1} \times X_m \mapsto Y_m, \, (x,\xi) \mapsto \bar{\d}(f\restr{U_{m+1}})(x)\cdot\xi$ is continuous.
	\end{enumerate}
	If these conditions hold, the $\scale$\hyp{}derivative $\D f: U^1 \times X \mapsto Y$ is simply $\{(\D f)_m\}_{m\in \setNz}$.	
\end{proposition}
\begin{remark}
	If an $\scale^0$ map $f: U \mapsto Y$ is such that each $f\restr{U_m}: U_m \mapsto Y_m$ is $C^1$, $m\in\setNz$ then, since the bonding maps $U_{m+1} \hookrightarrow U_m$ are restrictions of linear continuous maps, we see that the conditions of \thref{prop:sc1_alternative} are fulfilled, where the extension $\bar{\d}(f\restr{U_{m+1}})(x)$ is simply the original derivative $\d(f\restr{U_m}: U_m \mapsto Y_m)(x)$. In particular, we thus observe that scale maps which are $C^1$ on each layer are $\scale^1$.
\end{remark}

\begin{proposition}[in {\cite[Proposition~2.2]{hofer10}}] \thlabel{prop:shift_scsmooth}
	Let $X$ and $Y$ be Banach scales, $U\subset X$ be open, and $k\in\setN$. If $f: U \mapsto Y$ is $\scale^k$, then $f\restr{U_1}: U^1 \mapsto Y^1$ is $\scale^k$ and $\D(f\restr{U_1}) = (\D f)\restr{\Tg U_1}: \Tg U^1 \mapsto Y^1$.
\end{proposition}

\begin{proposition}[Chain Rule {\cite[Theorem~2.16]{hofer06}}] \thlabel{prop:chain_scsmooth}
	Let $X$, $Y$ and $Z$ be Banach scales, $U\subset X$ and $V\subset Y$ be open, and $k\in\setN$. Assume we are given two $\scale^k$ maps $f: U \mapsto Y$ and $g: V \mapsto Z$ with $f(U) \subset V$. Then $g\circ f: U \mapsto Z$ is $\scale^k$ and
	\begin{equation}
		\Tg^m (g\circ f) = (\Tg^m g) \circ (\Tg^m f)
	\end{equation}
	for all $1\leq m \leq k$.
\end{proposition}
\begin{proof}
	The proposition was proved in \cite{hofer06} for $k=1$. For $k>1$, induction using the definition of an $\scale^k$ map gives the result.
\end{proof}

As an application of these results, we derive basic properties of scale differentiation and prove $\scale$\hyp{}smoothness of some template maps. The properties in question are similar to the ones in traditional calculus, and so is their derivation.
\begin{lemma} \thlabel{lem:prod_scsmooth}
	Let $X$, $Y$, $Z$ and $W$ be Banach scales, $U\subset X$ and $V \subset Y$ be open, and $k\in\setN$. If $f:U \mapsto Z$ and $g: V \mapsto W$ are $\scale^k$, then $f\times g: U \times V \mapsto Z \times W$ is $\scale^k$ and $\D(f\times g) = \D f \times \D g: \Tg U \times \Tg V \iso \Tg(U\times V) \mapsto Z \times W$. Consequently, $\Tg(f\times g) = \Tg f \times \Tg g: \Tg U \times \Tg V \iso \Tg(U\times V) \mapsto \Tg(Z\times W) \iso \Tg Z \times \Tg W$.
\end{lemma}
\begin{proof}
	For $(x,y) \in X_1 \times Y_1$, we have
	\begin{multline}
		\frac{\|f\times g((x,y) + (t, s)) - f\times g(x,y) - \D f \times \D g((x,y),(t,s))\|_{Z_0\times W_0}}{\|(t,s)\|_{X_1\times Y_1}} \\
		\leq \frac{\|f(x+t)-f(x)-\D f(x,t)\|_{Z_0}}{\|t\|_{X_1}} + \frac{\|g(y+s)-g(y)-\D g(y,s)\|_{W_0}}{\|s\|_{Y_1}} \,,
	\end{multline}
	and the latter vanishes as $t \to 0$ in $X_1$ and $s \to 0$ in $Y_1$. Furthermore, the tangent map is $\Tg(f\times g) = \Tg f \times \Tg g$ which is $\scale^0$. Finally, by induction, if the statement is assumed for some $k\in\setN$ and $f$ and $g$ are $\scale^{k+1}$, then $\Tg(f\times g) = \Tg f \times \Tg g$ is $\scale^k$ by the induction hypothesis, hence $f\times g$ is $\scale^{k+1}$. This inductive method of proving smoothness is called \emph{bootstrapping}.
\end{proof}

\begin{lemma} \thlabel{lem:sc_properties}
	Let $X$, $Y$, $Z$ and $W$ be Banach scales, $U\subset X$ and $V\subset Y$ be open, and $k\in\setN$.
	\begin{enumerate}
		\item \label{lem:incl_scsmooth} The inclusion $\iota: U \mapsto X$ is $\scale$-smooth and $\D_x \iota = \id_{X_0}$ for all $x\in U_1$.
		\item \label{lem:const_scsmooth} For $y\in Y_\infty$, the constant map $\txt{const}_y: U \mapsto Y, \, x \mapsto y$ is $\scale$-smooth and $\D_x(\txt{const}_y) = 0$ for all $x\in U_1$.
		\item \label{lem:lin_scsmooth} An $\scale^0$\hyp{}operator $L: X \mapsto Y$ is $\scale$-smooth and $\D_x L = L \in \B(X_0,Y_0)$ for all $x\in X_1$.
		\item \label{lem:bilin_scsmooth} An $\scale^0$ bilinear map $B: X \times Y \mapsto Z$ is $\scale$-smooth and $\D_{(x,y)} B \cdot (\xi, \eta) = B(\xi, y) + B(x, \eta)$ for all $(x,y)\in X_1 \times Y_1$ and $(\xi,\eta)\in X_0\times Y_0$.
		\item \label{lem:restr_scsmooth} (Coordinate Restriction) If $f: U \times V \mapsto Z$ is $\scale^k$ and $y\in V_\infty$, then $f(\cdot, y): U \mapsto Z$ is $\scale^k$ and $\D[f(\cdot,y)] = \D f(\cdot,y, \cdot, 0): \Tg U \mapsto Z$.
		\item \label{lem:diagonal_prod_scsmooth} (Diagonal Product) If $f: U \mapsto Y$ and $g: U \mapsto Z$ are $\scale^k$, then $f \diagprod g: U \mapsto Y \times Z, \, x \mapsto (f(x), g(x))$ is $\scale^k$ and $\D_x (f \diagprod g) = \D_x f \diagprod \D_x g$ for all $x\in U_1$.
		\item \label{lem:sck_linear} (Linearity of the $\scale$-Derivative) $\scale^k(U,Y)$ is a linear subspace of $\scale^0(U,Y)$ and $\D: \scale^1(U,Y) \mapsto \scale^0(\Tg U, Y), \, f \mapsto \D f$ is linear. Hence, $\Tg: \scale^k(U,Y) \mapsto \scale^{k-1}(\Tg U, \Tg Y)$ is linear.
		\item \label{lem:sc_leibniz} (Leibniz Rule) If $f: U \mapsto Y$ and $g: U \mapsto Z$ are $\scale^k$ and $B=\cdot: Y \times Z \mapsto W$ is $\scale^0$ bilinear, then $f\cdot g: U \mapsto W, \, x \mapsto f(x)\cdot g(x)$ is $\scale^k$ and $\D (f\cdot g)(x,\xi) = \D f(x,\xi) \cdot g(x) + f(x) \cdot \D g(x,\xi)$ for all $x\in U_1$ and $\xi\in X_0$.
	\end{enumerate}
\end{lemma}
\begin{proof}
	Beginning with \ref{lem:lin_scsmooth}, we have $\|L(x+h) - L(x) - \D L(x, h)\|_{Y_0} = 0$ for all $x,h\in X_1$ by linearity, hence the Frèchet condition is satisfied. The tangent map is given by $\Tg L = L\restr{X_1} \times L: X^1 \times X \mapsto Y^1 \times Y$, which is $\scale^0$. To prove smoothness, we again use bootstrapping: if $L$ is $\scale^k$ for $k\in\setN$, then $\Tg L$ is also $\scale^k$ by \thref{lem:prod_scsmooth,prop:shift_scsmooth}, whence $L$ is $\scale^{k+1}$. In other words, $\Tg L$ is as smooth as $L$ is. One can prove \ref{lem:incl_scsmooth} and \ref{lem:const_scsmooth} using the same methodology.
	
	For \ref{lem:diagonal_prod_scsmooth}, we note that $f \diagprod g = (f \times g) \circ \diag_X\restr{U}$, where $\diag_X: X \mapsto X \times X,\, x \mapsto (x,x)$ is the ($\scale^0$ linear) diagonal map, and use the chain rule (\thref{prop:chain_scsmooth}), \thref{lem:prod_scsmooth} and \thref{lem:sc_properties}\ref{lem:incl_scsmooth},\ref{lem:lin_scsmooth}. Similarly, regarding \ref{lem:sck_linear}, $\scale^k(U,Y)$ is an (additive) subgroup of $\scale^0(U,Y)$ and the derivative is additive, since for $\scale^k$ maps $f,g: U \mapsto Y$, $f+g$ factors as $a \circ (f \diagprod g)$, where $a: Y \times Y \mapsto Y$ is the ($\scale^0$ linear) vector space addition. Regarding \ref{lem:restr_scsmooth}, it is clear that $y\in V_\infty$ induces a scale map $f(\cdot, y)$, since $V_\infty \subset V_m$ for $m\in\setNz$ and $f$ is a scale map. This map factors as $f \circ (\id_U \diagprod \txt{const}_y)$.
	
	The Frèchet condition of \ref{lem:bilin_scsmooth} is satisfied, since for some $K>0$
	\begin{multline}
		\frac{\|B(x+t, y+s) - B(x,y) - B(t,y) - B(x,s)\|_{Z_0}}{\|(t,y)\|_{X_1\times Y_1}} \leq \\
		\leq K \left(\frac{1}{\|t\|_{X_1}} + \frac{1}{\|s\|_{Y_1}}\right)^{-1} \,,
	\end{multline}
	and the latter vanishes as $t \to 0$ in $X_1$ and $s \to 0$ in $Y_1$. Again, $\Tg B$ is as smooth as $B$ is, since it is a composition of above proven constructions derived from $B$. The Leibniz rule is an easy corollary stemming from $f\cdot g = B \circ (f \diagprod g)$ and the homogeneity of \ref{lem:sck_linear} is a corollary thereof with $B: \setF \times Y \mapsto Y$ given by the scalar multiplication of $Y$ and the maps $\txt{const}_\alpha: U \mapsto \setF$ for $\alpha\in\setF$ and $f: U \mapsto Y$.	
\end{proof}
\begin{corollary} \thlabel{cor:tf_df_sck}
	Let $X$ and $Y$ be Banach scales, $U\subset X$ be open, and $f: U \mapsto Y$ be an $\scale^1$ map. Then for $k \in \setNz$, $\Tg f$ is $\scale^k$ if and only $\D f$ is $\scale^k$.
\end{corollary}
\begin{proof}
	We have $\D f = \pr^{\Tg Y}_2 \circ \Tg f: \Tg U \mapsto Y$ and $\Tg f = ((f\restr{U_1}\circ\pr^{\Tg U}_1) \times \D f) \circ \diag_{\Tg U}: \Tg U \mapsto \Tg Y$, where $\pr^{\Tg Y}_2: Y^1 \times Y \mapsto Y$ and $\pr^{\Tg U}_1: U^1 \times X \mapsto U^1$ are the canonical projections.
\end{proof}

For more easily dealing with scale maps defined on a product scale, and also for expressing flow equations on scales in Section~\ref{sec:hampde}, it is important to have a notion of partial differentiation. The most elegant solution consists of extending scale calculus to finitely\hyp{}indexed Banach scales, subsequently regarding a partial $\scale$\hyp{}derivative as the usual $\scale$\hyp{}derivative with respect to one argument while keeping the other argument fixed at a point of certain regularity. Such an extension of the theory is needed since, keeping the notation of \thref{lem:sc_properties}, for a scale map $f: U \times V \mapsto Z$ and a fixed element $y\in V_m$, $m\in\setN$, the scales involved on the map $f(\cdot, y)$ are indexed on $\{0,1,\ldots,m\}$ only. The extension is presented in \cite{crespo17}, but to avoid the overhead, we opt for an ad hoc definition here which gives the same result when the map in question is jointly $\scale^1$ {\it a priori}.

For an $\scale^1$ map $f: U \times V \mapsto Z$ (notation of \thref{lem:sc_properties}), we define the partial $\scale$\hyp{}derivative with respect to the first argument as the $\scale^0$ map $\frac{\partial f}{\partial x} \eqdef \D f(\cdot, \cdot, \cdot, 0): U^1 \times V^1 \times X \mapsto Z$. Similarly, with respect to the second argument, we have the $\scale^0$ map $\frac{\partial f}{\partial y} \eqdef \D f(\cdot, \cdot, 0, \cdot) : U^1 \times V^1 \times Y \mapsto Z$. Note that according to the coordinate restriction property of \thref{lem:sc_properties}\ref{lem:restr_scsmooth}, for every $y \in V_\infty$, $f(\cdot, y)$ is an $\scale^1$ map and $\D [f(\cdot, y)](x,\xi) = \frac{\partial f}{\partial x}(x,y,\xi)$ for every $x\in U_{m+1}$ and $\xi \in X_m$, $m\in\setNz$ (and similar formulas hold for $f(x, \cdot)$, $x\in U_\infty$).

As a last examination of the properties of $\scale$\hyp{}smoothness, we now turn our attention to curves, where the domain scale $X=\setR$ is the real one\hyp{}dimensional constant scale. This kind of maps will be used when dealing with $\scale$\hyp{}smooth flows in Section~\ref{sec:hampde}. To begin with, note that for a Banach scale $Y$ and $A\subset \setR$, scale ($\scale^0$) maps $A \mapsto Y$ are in 1-1 correspondence to plain one\hyp{}layer (continuous) maps $A \mapsto Y_\infty$, since all elements of $A$ are smooth. The correspondence is obtained by sending a scale ($\scale^0$) map $u: A \mapsto Y$ to its limit map $\lim_{k\in \setNz} u_k : A \mapsto Y_\infty$. Furthermore, for a Banach scale $Z$ and $W\subset Z$ open, scale maps $g: W \times \setR \mapsto Y$ which are linear in the second argument are in 1-1 correspondence with scale maps $\bar g: W \mapsto Y$ by sending $g$ to $\bar g = g(\cdot, 1)$, and $g$ is $\scale^k$ if and only if $\bar g$ is so, $k\geq 0$. This allows us to simplify the treatment of the $\scale$\hyp{}derivative of $\scale^1$ maps $u: U\subset\setR \mapsto Y_\infty$ by analysing $\D u(\cdot, 1): U \mapsto Y_{\infty}$ instead of the complete $\D u$. The following proposition, which ends this subsection, relates classical smoothness\footnote{One can still consider smoothness of curves into Frèchet spaces: a map $u: U\subset\setR \mapsto Y_\infty$ is continuously differentiable if (as usual) a continuous map $\dot u: U \mapsto Y_\infty$ exists with $\frac{u(t+h) - u(t)}{h} \to \dot u(t)$ in $Y_\infty$ as $h\to 0$, for all $t\in U$. $u$ is recursively defined to be $C^{k+1}$ if $\dot u$ is $C^{k}$, $k\geq 1$.} with $\scale$\hyp{}smoothness for these maps.
\begin{proposition} \thlabel{prop:domain_R_scsmoothness}
	Let $Y$ be a Banach scale, $U\subset \setR$ be open, and $k \in \setN$. Then a scale map $u: U\subset\setR \mapsto Y_\infty$ is $\scale^k$ if and only if it is $C^k$. If this holds, the $\scale$\hyp{}derivative $\D u(\cdot, 1): U \mapsto Y_{\infty}$ and the classical derivative $\dot u: U \mapsto Y_{\infty}$ coincide.
\end{proposition}
\begin{proof}
	Beginning with $k=1$, by \thref{prop:sc1_alternative}, $u: U \mapsto Y$ is $\scale^1$ if and only if $u: U \mapsto Y_m$ is $C^1$ for all $m\in\setNz$. If this holds, the $\scale$\hyp{}derivative is simply the classical derivative $\dot u: U \mapsto Y_{\infty}$ upon the above described identifications. In turn, all $u: U \mapsto Y_m$ being $C^1$ is equivalent to $u: U \mapsto Y_{\infty}$ being $C^1$, since the bonding maps $\iota_{\infty m}^Y$ are linear continuous. Note here that the topology of $Y_\infty$ is generated by the norms of $Y_m$, $m\geq 0$. Induction completes the argument for $k > 1$.
\end{proof}

\subsection{Scale Manifolds} \label{sec:sc_manifolds}
In a way similar to standard differential geometry \cite{zeidler97,lee12}, we can generalize $\scale$\hyp{}smoothness to topological spaces which are locally modeled on a Banach scale. We begin by introducing the notion of an $\scale^\infty$\hyp{}manifold \cite{hofer06,hofer10} and show subsequently that every $\scale^\infty$\hyp{}manifold gives rise to a filtration induced by the local scale structure. Next, we define the tangent scale at each point of a filtration subspace: a finitely\hyp{}indexed Banach scale for which each choice of a coordinate chart gives rise to an isomorphism to the local model. Finally, we define tangent bundles, $\scale$\hyp{}smooth maps between $\scale^\infty$\hyp{}manifolds and corresponding tangent maps. Recall that $m = \{0, 1, \ldots, m-1\}$ can be seen as a set for $m\in\setN$, and that for this section, Banach scales are assumed to be over the reals and on $\setNz$. 
	
\begin{definition} \thlabel{def:sc_manifold}
	Let $M$ be a Hausdorff space and $X$ be a Banach scale.
	\begin{enumerate}
		\item A coordinate chart is a pair $(U, \phi)$, where $U \subset M$ is an open subset and $\phi: U \isomapsto \phi(U) \subset X_0$ is a homeomorphism onto an open subset $\phi(U)$ of $X_0$. $U$ is called the coordinate domain of $(U, \phi)$, whereas $\phi$ is called the coordinate map.
		\item Two coordinate charts $(U, \phi)$ and $(U',\psi)$ are called $\scale^\infty$\hyp{}compatible whenever $U\cap U' = \emptyset$ or the transition map $\psi \circ \phi^{-1}: \phi(U\cap U') \mapsto \psi(U\cap U')$ is an $\scale^\infty$\hyp{}diffeomorphism.
		\item An $\scale^\infty$\hyp{}atlas is a collection $\mathcal{A} = \{(U_a, \phi_a)\}_{a\in A}$ of coordinate charts such that $\cup_{a\in A} U_a = M$ and $(U_a, \phi_a)$ is $\scale^\infty$\hyp{}compatible with $(U_b, \phi_b)$ for all $a,b\in A$.
		\item An $\scale^\infty$\hyp{}atlas $\mathcal{A}$ is called maximal whenever for all $\scale^\infty$\hyp{}atlases $\mathcal{B}$ with $\mathcal{A} \subset \mathcal{B}$, we have $\mathcal{B} = \mathcal{A}$.
		\item An $\scale$\hyp{}smooth structure for $M$ is a maximal $\scale^\infty$\hyp{}atlas $\mathcal{A}$ of charts from open subsets of $M$ onto open subsets of $X_0$ as above. The pair $(M, \mathcal{A})$ is called an $\scale^\infty$\hyp{}manifold (locally) modeled on $X$. Usually, $\mathcal{A}$ is suppressed from notation.
	\end{enumerate}
\end{definition}
	
\begin{remark} \thlabel{rmk:sc_mflds}
	\begin{enumerate}
		\item Replacing $\infty$ by $k$, one has the definition of an $\scale^k$\hyp{}manifold, $k\in\setNz$. Note that $\scale^k$\hyp{}manifolds are Banach manifolds as well.
		\item With the same proof as in the finite\hyp{}dimensional case, an $\scale^\infty$\hyp{}atlas $\mathcal{A}$ admits a unique maximal atlas $\overline{\mathcal{A}}$ with $\mathcal{A} \subset \overline{\mathcal{A}}$ \cite[Proposition~1.17(a)]{lee12}. Hence, for defining an $\scale^\infty$\hyp{}manifold, we only need to specify some (eventually non\hyp{}maximal) atlas.
		\item \label{rmk:sc_mflds_basis} It follows from the definition of an $\scale^\infty$\hyp{}manifold that a basis for the topology of $M$ is given by sets of the form $\phi^{-1}(V)$, where $(U,\phi)$ is a coordinate chart and $V\subset \phi(U)$ is open.
		\item \label{rmk:sc_mflds_sets} If $M$ is a set instead of a topological space, one can modify the definition requiring, for each coordinate chart $(U, \phi)$, the coordinate domain $U$ only to be a subset of $M$, $\phi$ to be a bijection onto an open subset $\phi(U)$ of $X_0$, and $\phi(U \cap U') \subset \phi(U)$ to be open for all coordinate domains $U'\subset M$, while maintaining the $\scale^\infty$ compatibility condition. The basis in \ref{rmk:sc_mflds_basis} then generates a unique topology on $M$ such that the coordinate domains are open in $M$ and the coordinate maps are homeomorphisms. Consequently, if $M$ endowed with this topology is Hausdorff, the coordinate charts induce an $\scale$\hyp{}smooth structure on $M$.
		\item It is possible to change the topological requirements for $\scale^\infty$\hyp{}manifolds. We restrain ourselves to the fairly weak Hausdorff condition, since it is easily shown to be inherited by each member of the filtration induced by an $\scale^\infty$\hyp{}manifold (\thref{lem:filtration_sc_mfld}); stronger conditions are more difficult to handle.
	\end{enumerate}
\end{remark}
		
If $M$ is an $\scale^\infty$\hyp{}manifold, $(U, \phi)$ is a coordinate chart and we define $V \eqdef \phi(U) \subset X_0$, the open system of subsets induced by $V$, $V_m = V \cap X_m$, $m\in\setNz$, pulls back to a descending filtration on $U$, $U_m \eqdef \phi^{-1}(V_m)$, $m\in\setNz$. An $\scale^\infty$\hyp{}atlas $\mathcal{A} = \{(U_a, \phi_a)\}_{a\in A}$ for the $\scale$\hyp{}smooth structure of $M$ then induces a global filtration $M_m \eqdef \cup_{a\in A} (U_a)_m$, $m\in\setNz$. Since the transition maps of $\mathcal{A}$ are scale maps, we have $U_m \cap U' = U \cap U_m' = U_m \cap U_m'$ and the filtration is independent of the chosen atlas. For a general $U\subset M$ open, we may then set $U_m \eqdef U \cap M_m$, $m\in\setNz$, and this definition is consistent with the case that $U$ is a coordinate domain.
	
Each coordinate chart $\phi: U \isomapsto V \subset X_0$ of $M$ restricts to bijections $\phi_m \eqdef \phi\restr{U_m}: U_m \isomapsto V_m$. 
We claim that for an $\scale^\infty$\hyp{}atlas $\mathcal{A}$ as above and fixed $m>0$, the bijections $\{(\phi_a)_m\}_{a\in A}$, induce an $\scale$\hyp{}smooth structure on $M_m$ with local model $X^m$. It is then clear that the filtration of $M_m$ is given by $(M_m)_l = M_{m+l}$, $l\in\setNz$. In analogy to the linear case of a Banach scale, we define the shifted filtration $M^m \eqdef \{M_{m+l}\}_{l\in\setNz}$ and the limit $M_\infty \eqdef \cap_{m\in\setN} M_m$ with limit topology.
\begin{lemma} \thlabel{lem:filtration_sc_mfld}
	Let $M$ be an $\scale^\infty$\hyp{}manifold, $\mathcal{A} = \{(U_a, \phi_a)\}_{a\in A}$ be an $\scale^\infty$\hyp{}atlas for the $\scale$\hyp{}smooth structure of $M$ and $M_m = \cup_{a\in A} (U_a)_m$, $m\in\setNz$. Then, for each $m\in\setNz$, $M_m$ is an $\scale^\infty$\hyp{}manifold with local model $X^m$ and coordinate charts $(\phi_a)_m: (U_a)_m \isomapsto (V_a)_m$, $a\in A$. Furthermore, the inclusions $M_m \hookrightarrow M_l$ are continuous for all $m>l\in\setNz$.
\end{lemma}
\begin{proof}
	Only the case $m>0$ is new. We aim to prove the statement by using \thref{rmk:sc_mflds}\ref{rmk:sc_mflds_sets} for the set $M_m$, the bijections $(\phi_a)_m$, $a\in A$, and the local model $X^m$. First, let $\mathcal{A} \owns \phi: U \isomapsto V \subset X_0$ be a coordinate chart with induced bijections $\phi_m: U_m \isomapsto V_m$, $m\in\setNz$, and $U' \in \mathcal{A}$ be a further coordinate domain. One has $\phi_m(U_m \cap U') = \phi(U\cap U') \cap V_m$ 
	and since $U_m \cap U' = U_m \cap U_m'$, the continuity of the inclusion map $V_m \hookrightarrow V$ implies that $\phi_m(U_m\cap U_m')$ is open in $V_m$. Hence, the bijections $(\phi_a)_m$ induce a topology on $M_m$ with basis given by sets of the form $(\phi_a)_m^{-1}(W)$ with $W\subset (V_a)_m$ open\footnote{we stress the fact that $(V_a)_m$ inherits the (finer) topology of $X_m$ and not the one of $V_a \subset X_0$.} and $a\in A$.
		
	To see that the inclusions $M_m \hookrightarrow M_l$ are continuous for $m > l \in\setNz$, let $(U,\phi)\in\mathcal{A}$, $W\subset V_l$ be open, and $x \in \phi_l^{-1}(W) \cap M_m$. Then $x\in U_m$, hence $x\in \phi_m^{-1}(W\cap V_m) \subset \phi_l^{-1}(W) \cap M_m$. Again using the continuity of the inclusion $V_m \hookrightarrow V_l$, we conclude that $x$ is an inner point. Now, the fact that $M_m \hookrightarrow M_0$ is continuous precisely means that the topology on $M_m$ is finer than the subspace topology induced by $M_0$. The Hausdorff property of $M_m$ then follows directly from the fact that $M=M_0$ is Hausdorff.
\end{proof}
	
To complement the basic definitions, we give simple constructions of $\scale^\infty$\hyp{}manifolds which work the same way as in the finite\hyp{}dimensional case and, subsequently, we introduce the pivotal example of an $\scale^\infty$\hyp{}manifold which will be used when discussing Hamiltonian flows in Section~\ref{sec:hampde}. The proof of the constructions is simple and left to the reader.
\begin{lemma} \thlabel{lem:sc_mfld_constructions}
	Let $X$ and $Y$ be Banach scales, $M$ and $N$ be $\scale^\infty$\hyp{}manifolds locally modeled on $X$ and $Y$, and $\mathcal{A} = \{(U_a,\phi_a)\}_{a\in A}$ and $\mathcal{B} = \{(\Omega_b,\psi_b)\}_{b\in B}$ be atlases for $M$ and $N$, respectively.
	\begin{enumerate}
		\item $X$ is by itself an $\scale^\infty$\hyp{}manifold: the single chart $\id_{X_0}: X_0 \mapsto X_0$ defines an $\scale$\hyp{}smooth structure on $X_0$ with local (global!) model $X$ and filtration $(X_0)_m = \iota_{m0}(X_m) \iso X_m$, where $\iota_{m0}: X_m \hookrightarrow X_0$ is the bonding map. Clearly, the $\scale$\hyp{}smooth structure of $(X_0)_m$ is the one obtained by applying this construction to $X^m$.
		\item An open subset $U\subset M$ has an $\scale$\hyp{}smooth structure given by the charts $\phi_a\restr{U \cap U_a}: U \cap U_a \isomapsto \phi_a(U \cap U_a)$, $a\in A$, local model $X$ and filtration $U_m = U \cap M_m$, $m\in\setNz$.
		\item The product $M\times N$ has an $\scale$\hyp{}smooth structure given by the charts $\phi_a \times \psi_b: U_a \times \Omega_b \isomapsto \phi_a(U_a) \times \psi_b(\Omega_b)$, $(a,b)\in A\times B$, local model $X\times Y$ and filtration $(M\times N)_m = M_m \times N_m$, $m\in\setNz$.
	\end{enumerate}
\end{lemma}

\begin{example} \thlabel{ex:projective_hilbert}
	If $X$ is a separable Hilbert space over $\setF = \setR~\txt{or}~\setC$ and $\scf{X}$ is its induced scale (restricted to indices in $\setNz$ and scalars in $\setR$), then the projectivization of $X=X_0$,
	\begin{equation}
		M \eqdef \projvect(X) = \faktor{X \setminus \{0\}}{\setF^*} \,,
	\end{equation}
	is an $\scale^\infty$\hyp{}manifold modeled on $\scf{X}$. To prove this, first note that the isometric isomorphism $X \isomapsto l_\setF^2$ descends to a homeomorphism $\projvect(X) \isomapsto \projvect(l_\setF^2)$, hence $X=l_\setF^2$ without loss of generality. Similarly to the finite\hyp{}dimensional case of $\projvect(\setF^{d}) = \setF \projvect^{d-1}$, we can define $U_a \eqdef \{[x] \in \projvect(X): x_a \neq 0\} \subset \projvect(X)$ open, $a\in\setZ$, and coordinate charts $\phi_a: U_a \isomapsto X_0$ given by
	\begin{equation} \label{eq:charts_projective_hilbert}
		\phi_a([x])_n = \frac{1}{x_a} \cdot
		\begin{cases}
			x_n		&\txt{if}~n<a \\
			x_{n+1}	&\txt{if}~n\geq a \,.
		\end{cases}
	\end{equation}
	These charts are easily seen to be homeomorphisms and the corresponding transition maps are $\scale$\hyp{}smooth.
	
	To obtain a natural filtration on $\projvect(X)$, we need to impose a slight condition on the defining sequences. Specifically, if both sequences $\{{\nu_n}/{\nu_{n+1}}\}_{n\in\setZ}$ and $\{{\nu_{n+1}}/{\nu_{n}}\}_{n\in\setZ}$ are bounded, we can define a compatible atlas for $\projvect(X_m)$ given by \eqref{eq:charts_projective_hilbert}, but as maps $\{[x] \in \projvect(X_m): x_a \neq 0\} \isomapsto X_m$. In that case, we have a homeomorphism
	\begin{equation}
		M_m \isomapsto \projvect(X_m), \, [x]_{\projvect(X_0)} \mapsto [x]_{\projvect(X_m)}
	\end{equation}
	which is actually an $\scale^\infty$\hyp{}diffeomorphism (see \thref{def:sck_maps_mfld}). The topology on $\projvect(X_m)$ can be seen to be separable metrizable (in particular Hausdorff) for all $m\in\setNz$.
\end{example}
	
As \thref{lem:filtration_sc_mfld} and the preceding discussion show, an $\scale^\infty$\hyp{}manifold can be re\hyp{}interpreted as a filtration of topological spaces which locally is levelwise homeomorphic to an open subset of the local model, and where the transition maps are $\scale$\hyp{}smooth. It is then natural to extend the definitions of Section~\ref{sec:banach_scales} involving scale maps to this context. A function $f: M \mapsto N$ between two $\scale^\infty$\hyp{}manifolds $M$ and $N$ modeled on $X$ and $Y$, respectively, is said to be a scale map if $f(M_m) \subset N_m$ for each $m\in\setNz$. In other words, we require $f$ to restrict to maps $f_m \eqdef f\restr{M_m}: M_m \mapsto N_m$. A scale map $f: M \mapsto N$ is called $\scale^0$ if all $f_m$ are continuous. It is an $\scale^0$\hyp{}homeomorphism if, additionally, it is bijective and $f^{-1}: N \mapsto M$ is $\scale^0$.

Following our developments of Section~\ref{sec:sc_calculus}, one also expects to define $\scale^k$ maps ($k>0$) in the manifold context. For this, we need the concept of tangent scales and bundles on $\scale^\infty$\hyp{}manifolds.
\begin{definition} \thlabel{def:tangent_scales}
	Let $M$ be an $\scale^\infty$\hyp{}manifold modeled on a Banach scale $X$ and let $p\in M_{m+1}$, $m\in\setNz$.
	\begin{enumerate}
		\item \label{def:sc_tg_spaces_set} For pairs $((U, \phi), v)$, where $(p\in U, \phi)$ is a coordinate chart of $M$ around $p$ and $v\in X_m$, define $((U, \phi), v)$ and $((U', \psi), w)$ to be equivalent if $\D(\psi\, \phi^{-1})(\phi(p), v) = w$. Define the $m\th$\hyp{}layer tangent space $(\Tg_p M)_m$ of $M$ at $p$ to be the corresponding quotient space.
		\item \label{def:sc_tg_spaces_banach} Endow $(\Tg_p M)_m$ with the structure of a Banachable (completely normable) space by pulling back the vector space operations and topology of $X_m$ via the well\hyp{}defined bijection $(\D_p\phi)_m: (\Tg_p M)_m \isomapsto X_m,\, [(U',\psi),w] \mapsto \D(\phi\, \psi^{-1})(\psi(p), w)$ for a given coordinate chart $(U(p),\phi)$.
		\item \label{def:sc_tg_scales} Since $p\in M_{m+1} \subset M_m \subset \ldots \subset M_1$, we have well\hyp{}defined $k\th$\hyp{}layer tangent spaces $(\Tg_p M)_k$, $0\leq k\leq m$. The maps $(\D_p\phi)_k$, $k\in m+1 = \{0, 1, \ldots, m\}$, induce the structure of a Banachable scale on $\Tg_p M \eqdef \{(\Tg_p M)_k\}_{k\in m+1}$ with bonding maps $(\Tg_p M)_k \hookrightarrow (\Tg_p M)_l, \, [(U, \phi), v] \mapsto [(U, \phi), \iota_{kl}(v)]$, $k>l \in m+1$, and with this structure, $\D_p\phi \eqdef \{(\D_p\phi)_k\}_{k \in m+1}: \Tg_p M \isomapsto X\restr{m+1}$ is an isomorphism of scales. The scale $\Tg_p M$ is called the tangent scale of $M$ at $p$.
	\end{enumerate}
\end{definition}
	
One can easily verify that the structures introduced in \thref{def:tangent_scales} \ref{def:sc_tg_spaces_banach},\ref{def:sc_tg_scales} are independent of the chosen coordinate chart $(U,\phi)$. We emphasize the fact that, as in the finite\hyp{}dimensional case, there is no preferred norm on $(\Tg_p M)_m$, hence the denomination of Banach\emph{able} spaces $(\Tg_p M)_m$ and scale $\Tg_p M$, where we only refer to their vector topology. Norms on $\Tg_p M$ only become defined when choosing coordinates and two coordinate charts induce equivalent norms. Note also that when $p\in M_\infty$, we obtain a tangent scale $\Tg_p M$ on $\setNz$ since $M_\infty \subset M_{m+1}$ for all $m\in\setNz$.

The constructions of \thref{lem:sc_mfld_constructions} have the tangent scales expected from finite dimensions, as the following lemma shows.
\begin{lemma} \thlabel{lem:tg_scales_constructions}
	Let $X$ and $Y$ be Banach scales, $M$ and $N$ be $\scale^\infty$\hyp{}manifolds locally modeled on $X$ and $Y$, respectively, and let also $m\in\setNz$. We have canonical isomorphisms of scales as follows.
	\begin{enumerate}
		\item \label{lem:tg_scale_flat} The tangent scale of the $\scale^\infty$\hyp{}manifold $X$ at $x\in X_{m+1}$ is $\Tg_x X \iso X\restr{m+1}$.
		\item \label{lem:tg_scale_open} The tangent scale of an open subset $U\subset M$ at $p\in U_{m+1}$ is $\Tg_p U \iso \Tg_p M$.
		\item \label{lem:tg_scale_product} The tangent scale of $M\times N$ at $(p,q)\in (M\times N)_{m+1}$ is $\Tg_{(p,q)} (M \times N) \iso \Tg_p M \times \Tg_q N$.
	\end{enumerate}
\end{lemma}
\begin{proof}
	The construction of the isomorphisms is straightforward. For example, for $p\in M_{m+1}$ and $q\in N_{m+1}$, coordinate charts $(U(p), \phi)$ and $(\Omega(q), \psi)$ of $M$ and $N$, respectively, give rise to isomorphisms of scales $\D_p \phi: \Tg_p M \isomapsto X\restr{m+1}$, $\D_q \psi: \Tg_q N \isomapsto Y\restr{m+1}$ and $\D_{(p,q)} (\phi\times\psi): \Tg_{(p,q)} (M\times N) \isomapsto (X \times Y)\restr{m+1}$, which combine to the isomorphism $(\D_p \phi \times \D_q \psi)^{-1}\circ \D_{(p,q)} (\phi\times\psi): \Tg_{(p,q)} (M \times N) \isomapsto \Tg_p M \times \Tg_q N, \, [(U\times\Omega, \phi\times\psi), (v,w)] \mapsto ([(U,\phi), v], [(\Omega,\psi), w])$. This isomorphism is easily seen to be independent of the chosen coordinate charts.
\end{proof}
		
By varying $p$, we construct the tangent bundle of $M$ as $\Tg M \eqdef \cup_{p\in M_1} \{p\} \times(\Tg_p M)_0$ with the canonical projection $\pi_{\Tg M}: \Tg M \mapsto M_1$, $(p, v) \mapsto p$. If for $U \subset M$ open we set $\Tg U \eqdef \pi_{\Tg M}^{-1}(U_1)$, each coordinate chart $\phi: U \isomapsto V \subset X_0$ of $M$ induces a bijection $\Tg \phi: \Tg U \isomapsto V_1 \times X_0, \, (p, v) \mapsto (\phi(p), \D_p \phi(v))$. Moreover, if $(U', \psi)$ is an additional coordinate chart, then $\Tg \phi(\Tg U \cap \Tg U') = \phi(U\cap U')_1 \times X_0$ is open in $V_1 \times X_0$, and since $\psi\circ\phi^{-1}: \phi(U\cap U') \isomapsto \psi(U\cap U')$ is $\scale$\hyp{}smooth, the transition map $\Tg \psi \circ (\Tg \phi)^{-1}: \phi(U\cap U')^1 \times X \isomapsto \psi(U\cap U')^1 \times X$ 
is $\scale$\hyp{}smooth as well. Consequently, an $\scale^\infty$\hyp{}atlas $\{(U_a, \phi_a)\}_{a\in A}$ for $M$ gives rise to an $\scale^\infty$\hyp{}atlas $\{(\Tg U_a, \Tg \phi_a)\}_{a\in A}$ for $\Tg M$, whence an $\scale$\hyp{}smooth structure with local model $X^1 \times X$. It is easy to see that the induced filtration is simply $(\Tg M)_m = \cup_{p\in M_{m+1}} \{p\} \times (\Tg_p M)_m$, where $(\Tg_p M)_m$ is considered as a subspace of $(\Tg_p M)_0$ via the corresponding bonding map. Moreover, the fiber $\pi_{\Tg M}^{-1}(p)$ over $p\in M_{m+1}$ is $\{p\} \times (\Tg_p M)_0 \iso (\Tg_p M)_0$, where we recover the scale structure of the tangent scale as $(\Tg_p M)_m \subset (\Tg_p M)_{m-1} \subset \ldots \subset (\Tg_p M)_0$.
	
Tangent scales and bundles allow us to formalize $\scale$\hyp{}smoothness on $\scale^\infty$\hyp{}manifolds by working locally, as the following definition shows.
\begin{definition} \thlabel{def:sck_maps_mfld}
	Let $M$ and $N$ be $\scale^\infty$\hyp{}manifolds modeled on $X$ and $Y$, respectively, $f: M \mapsto N$ be an $\scale^0$ map, and let $k\in\setN\cup\{\infty\}$.
	\begin{enumerate}
		\item \label{def:sc_k_map_mfld} The map $f$ is said to be $\scale^k$ if for each $p\in M$ there are charts $(U(p), \phi)$ of $M$ and $(\Omega(f(p)), \psi)$ of $N$ such that $f(U) \subset \Omega$ and $\psi\,f\restr{U}\,\phi^{-1}: \phi(U) \mapsto \psi(\Omega) \subset Y$ is $\scale^k$.
		\item \label{def:sc_derivative_mfld} If $f$ is $\scale^1$ we define, for each $p\in M_{m+1}$, $m\in\setNz$, the derivative of $f$ at $p$ to be the scale morphism $\D_p f: \Tg_p M \mapsto \Tg_{f(p)} N$ given by the diagram
			\[ \begin{tikzcd}[column sep=huge]
				\Tg_p M \ar[dashed]{r}{\D_p f} \ar{d}{\sim}[swap]{\D_p\phi}	& \Tg_{f(p)} N \ar{d}{\D_{f(p)}\psi}[swap]{\sim} \\
				X\restr{m+1} \ar{r}{\D_{\phi(p)}(\psi\,f\restr{U}\,\phi^{-1})}								& Y\restr{m+1}
			\end{tikzcd}\,,\]
		where the lower row is the $\scale$\hyp{}derivative of Section~\ref{sec:sc_calculus}. We define the tangent map $\Tg f: \Tg M \mapsto \Tg N, \, (p,v) \mapsto (f(p), \D_p f(v))$.
		\item The map $f$ is said to be an $\scale^k$ immersion (submersion) if it is $\scale^k$ and $\D_p f$ is injective (surjective) for all $p \in M_{m+1}$ and $m\in\setNz$. 
		\item The map $f$ is said to be an $\scale^k$\hyp{}diffeomorphism if it is $\scale^k$, bijective, and $f^{-1}: N \mapsto M$ is $\scale^k$.
	\end{enumerate}
\end{definition}
	
As expected, since the transition maps of $M$ and $N$ are $\scale^\infty$\hyp{}diffeomorphisms and since the chain rule of \thref{prop:chain_scsmooth} holds, \thref{def:sck_maps_mfld}\ref{def:sc_k_map_mfld},\ref{def:sc_derivative_mfld} is independent of the choice of charts $(U,\phi)$ and $(\Omega,\psi)$ satisfying $f(U) \subset \Omega$. Also, one could have dropped the umbrella assumption that $f$ is $\scale^0$, since this follows directly from part \ref{def:sc_k_map_mfld}. Some easy consequences of this definition are summarized in the following lemmas.
\begin{lemma}[Chain Rule]
	If $M$, $N$ and $P$ are $\scale^\infty$\hyp{}manifolds and $f: M \mapsto N$ and $g: N \mapsto P$ are $\scale^k$, $k\in\setN\cup\{\infty\}$, then $g\circ f: M \mapsto P$ is $\scale^k$ with $\D_p (g\circ f) = \D_{f(p)} g \circ \D_{p} f: \Tg_p M \mapsto \Tg_{g\circ f(p)} P$ for every $p\in M_{m+1}$ and $m\in\setNz$, hence $\Tg(g\circ f) = \Tg g \circ \Tg f: \Tg M \mapsto \Tg P$.
\end{lemma}
\begin{proof}
	Apply the chain rule for Banach scales in \thref{prop:chain_scsmooth}.
\end{proof}
\begin{lemma}
	Let $M$ and $N$ be $\scale^\infty$\hyp{}manifolds modeled on $X$ and $Y$, respectively.
	\begin{enumerate}
		\item \label{lem:coordinate_chart_scsmooth} A coordinate chart $\varphi: U \isomapsto V \subset X_0$ of $M$ is an $\scale^\infty$\hyp{}diffeomorphism.
		\item \label{lem:incl_scmfld_scsmooth} The inclusion maps $M^k \hookrightarrow M^l$ are injective $\scale$\hyp{}smooth immersions, $k>l\in\setNz$.
		\item \label{lem:tgbundle_proj_scsmooth} The projection $\pi_{\Tg M}: \Tg M \mapsto M^1$ is a surjective $\scale$\hyp{}smooth submersion.
		\item \label{lem:tangent_map_scsmooth} For an $\scale^1$ map $f: M \mapsto N$, $f$ is $\scale^k$ if and only if $\Tg f: \Tg M \mapsto \Tg N$ is $\scale^{k-1}$, $k\in\setN\cup\{\infty\}$.
	\end{enumerate}
\end{lemma}
\begin{proof}
	Part \ref{lem:coordinate_chart_scsmooth} essentially holds by definition: for every $p\in U$, we can use trivialize $U$ using $\varphi$ itself and $V$ using $\id_V$. In turn, the maps of \ref{lem:incl_scmfld_scsmooth} are locally given by the inclusions $V^k \hookrightarrow V^l$, where $\phi: U \isomapsto V \subset X_0$ is a coordinate chart, and these have $\scale$\hyp{}derivative $X^k \hookrightarrow X^l$ at every point in $V^{k+1}$. Similarly, the tangent bundle projection is locally given by the projection $V^1 \times X \mapsto V^1$. Finally, the tangent map of an $\scale^1$ map $f$ is given around $p\in M$ by $V^1 \times X \mapsto W^1 \times Y, \, (x,v) \mapsto (\psi\, f\restr{U}\, \phi^{-1}(x), \D(\psi\, f\restr{U}\, \phi^{-1})(x,v))$, where $\phi: U(p) \isomapsto V \subset X_0$ and $\psi: \Omega(f(p)) \isomapsto W \subset Y_0$ are charts with $f(U) \subset \Omega$.
\end{proof}

Generalizing partial differentiation of Section~\ref{sec:sc_calculus} is straightforward: for $\scale^\infty$\hyp{}manifolds $M$, $N$ and $P$, an $\scale^1$ map $f: M \times N \mapsto P$ and $(p,q)\in M_{m+1} \times N_{m+1}$, $m\in\setNz$, the partial derivative at $(p,q)$ with respect to the first argument is
\begin{equation}
	\frac{\partial f}{\partial p}(p,q) \eqdef \D_{(p,q)} f (\cdot, 0): \Tg_p M \mapsto \Tg_{f(p,q)} P \,,
\end{equation}
where we use the identification of \thref{lem:tg_scales_constructions}\ref{lem:tg_scale_product}. Naturally, a similar formula holds for differentiation with respect to the second argument. We also remark that when $M=\setR$, we regard $\frac{\partial f}{\partial p}(p,q)$ as an element of $(\Tg_{f(p,q)} P)_m$ by noting $\Tg_p \setR \iso \setR$ and applying 1 to the above map of scales.

As a final addendum to this section, we define sections of the tangent bundle (vector fields) in a similar manner as in the finite\hyp{}dimensional case.
\begin{definition} \thlabel{def:sections_sc_tg_bundle}
	An $\scale^k$ section of the tangent bundle of an $\scale^\infty$\hyp{}manifold $M$, $k\in\setNz\cup\{\infty\}$, is an $\scale^k$ map $s: M^1 \mapsto \Tg M$ with $\pi_{\Tg M}\circ s = \id_{M^1}$.
\end{definition}

\section{Hamiltonian Partial Differential Equations} \label{sec:hampde}

In this section, we carry the concepts of Hamiltonian vector fields and flows in finite\hyp{}dimensional symplectic geometry over to the $\scale$\hyp{}calculus framework. We first introduce the relevant concepts in the linear case of a Banach scale while being guided by the prototypical example of the free Schrödinger equation. Eventually, we arrive at the conclusion that an extension of $\scale$\hyp{}calculus is needed to handle Hamiltonian maps: strong $\scale$\hyp{}smoothness. After an extensive motivation, we define this new concept and show that it is invariant under pre\hyp{}composition with $\scale$\hyp{}smooth symplectomorphisms. This property makes strong $\scale$\hyp{}smoothness amenable to being used with $\scale^\infty$\hyp{}manifolds. All Banach scales in the presentation are assumed to be over the real numbers, or restricted by scalars when necessary.

After presenting the linear case, we generalize to $\scale^\infty$\hyp{}manifolds. We introduce symplectic $\scale^\infty$\hyp{}manifolds by restricting the atlas so that the transition maps are symplectic. Once a symplectic $\scale^\infty$\hyp{}manifold is given, we may extend the tangent structure of the manifold to support a symplectic form on each tangent scale, and we may also construct a cotangent bundle. We outline the necessary backbone for this, subsequently generalizing $\scale$\hyp{}smoothness, Hamiltonian vector fields and flows to this non-linear case. We illustrate the use of the new concepts with the free Schrödinger equation on a projective Hilbert space.

\subsection{Flows on Banach scales} \label{sec:banach_scales_flows}
Naturally, one needs to understand vector fields and flows on Banach scales before discussing \emph{Hamiltonian} vector fields and flows. We define and use these to formalize the free Schrödinger equation. For simplicity, we only work with complete autonomous vector fields in this paper. That is to say, we assume the vector fields to be time-independent and that their integral curves exist on the complete real line.

Let $X$ be a Banach scale on $\setNz$. We define an autonomous $\scale$\hyp{}smooth vector field on $X$ to be an $\scale^\infty$ map $V: X^1 \mapsto X$. We use the shifted scale $X^1$ for the domain, since the vector fields we are interested in Hamiltonian PDEs are densely defined ({\it e.g.}, see the prototypical \thref{ex:flow_free_schrodinger} further on). $V$ is said to be complete, or to have a global flow, if there exists an $\scale$-smooth map $\varphi: \setR \times X \mapsto X$ such that
\begin{align} \label{eq:sc_flow_equations}
	\frac{\partial \varphi}{\partial t}(t,v) &= V \circ \varphi(t,v) \,, &
	\varphi(0,v) &= v
\end{align}
for all $t \in\setR$ and $v\in X_{m+1}$, $m\in\setNz$, where $\frac{\partial \varphi}{\partial t}: \setR \times X^1 \mapsto X$ is the partial derivative of $\varphi$ in the scale sense. Clearly, it is sufficient that the equality holds for $v\in X_1$. Note here that the vector field can be recovered as $V = \frac{\partial \varphi}{\partial t}(0,\cdot)$. Moreover, taking \thref{prop:domain_R_scsmoothness} into account, if a global flow $\varphi$ exists, then the initial value problem
\begin{align} \label{eq:ivp_sc_flow}
	\frac{\d u}{\d t}	&= V \circ u: \setR \mapsto X_\infty \,, &
				u(0)	&= u_0
\end{align}
for a smooth unknown $u: \setR \mapsto X_{\infty}$ and initial condition $u_0 \in X_{\infty}$ has solution given by $t \mapsto \varphi(t,u_0)$.

\begin{example} \thlabel{ex:flow_free_schrodinger}
	The free Schrödinger equation
	\begin{equation}
		\i u_t = - \Delta u
	\end{equation}
	for $u: \setR \times S^1 \mapsto \setC, \, (t,x) \mapsto u(t,x)$, where $\Delta = (\cdot)_{xx}$ is the Laplacian operator, can be rewritten in evolution form by taking the (double\hyp{}spaced) Levi\hyp{}Sobolev scale $X_s = W^{2s,2}(S^1, \setC)$, $s\in\setNz$, and defining the vector field $V: X^1 \mapsto X, \, u \mapsto \i \Delta u$. Here, $X$ is seen as a real scale and the Laplacian $\Delta: X_{s+1} \mapsto X_s$ is taken in the weak sense, corresponding to the Fourier multiplier $[n \mapsto (\i n)^2] \in \setC^\setZ$. The evolution equation then simply reads $\dot u = V(u), ~ u(0) = u_0$ for a smooth curve $u: \setR \mapsto X_\infty$ and $u_0\in X_\infty$.
	
	We claim that the vector field $V$ is complete with $\scale$\hyp{}smooth flow given by
	\begin{equation} \label{eq:flow_free_schrodinger}
		\varphi: \setR \times X \mapsto X, \, (t,u) \mapsto \e^{\i t\Delta} u \,,
	\end{equation}
	where $\e^{\i t\Delta}$ is the bounded linear $X_s$ operator with Fourier multiplier $\e^{-\i t n^2}$. To see that $\varphi$ is $\scale^0$ to begin with, identify the scale $X$ with $\{l_\setC^{2,2s}\}_{s\in\setNz}$ using the Fourier series. For $x\in l^{2,2s}$ fixed and $N\in\setN$, the function $\setR \mapsto \setC^{2N-1} \subset l^{2,2s}, \, t \mapsto \{\e^{-\i t n^2} x_n\}_{|n| < N}$ is continuous and the expression
	\begin{equation}
		\sup_{t\in\setR} \| \{\e^{-\i t n^2} x_n\}_{|n| < N} - \varphi(t,x) \|_{l^{2,2s}}
		= \sup_{t\in\setR} \left(\sum_{|n|\geq N} |\e^{-\i t n^2} x_n|^2 (1+n^2)^{2s}\right)^\frac{1}{2}
	\end{equation}
	vanishes as $N\to\infty$, from where $\varphi(\cdot, u): \setR \mapsto X_s$ is continuous for all $u\in X_s$ by the uniform limit theorem. To prove joint continuity of $\varphi: \setR \times X_s \mapsto X_s$, just note that $\varphi$ is linear in the second argument, that
	\begin{multline}
		\|\varphi(t,u) - \varphi(t_0, u_0) \|_{X_s} \leq \\
		\leq \|\varphi(t,\cdot)\|_{\B(X_s)} \|u-u_0\|_{X_s} + \|\varphi(t,u_0) - \varphi(t_0, u_0) \|_{X_s}
	\end{multline}
	for $t, t_0\in\setR$ and $u,u_0\in X_s$, and that $\varphi(t,\cdot)$ is uniformly bounded in $\B(X_s)$ (by 1).

	As to the $\scale$\hyp{}smoothness claim, note that $\varphi(\cdot, u): \setR \mapsto X_0$ is $C^1$ for $u\in X_1$ with derivative $\d \varphi(\cdot, u)/\d t = \i \Delta \varphi(t, u)$. From this and the above mentioned linearity follows the suggestive $\scale$-derivative
	\begin{equation} \label{eq:dflow_free_schrodinger}
		\D \varphi: \setR \times X^1 \times \setR \times X \mapsto X ,\, (t, u, h, \xi) \mapsto \varphi(t,\xi) + h\, \i \Delta \varphi(t, u) \,.
	\end{equation}
	The required Frèchet condition of \eqref{eq:frechet_condition} is satisfied: for $t,h\in \setR$ and $u,\xi \in X_1$, we see in the estimate
	\begin{multline} \label{eq:frechet_flow_schrodinger}
		\frac{\| \varphi(t + h, u + \xi) - \varphi(t,u) - \D \varphi(t,u,h,\xi) \|_{X_0}}{|h| + \|\xi\|_{X_1}} \\
		\leq \left\| \frac{\varphi(t + h, u) - \varphi(t,u)}{h} - \i \Delta \varphi(t, u) \right\|_{X_0} + \| \varphi(t + h, \cdot) - \varphi(t,\cdot) \|_{\B(X_1,X_0)}
	\end{multline}
	that the first term vanishes as $h\to 0$ by the differentiability of $\varphi(\cdot, u)$, and that the second term vanishes by the compactness of the embedding $X_1 \hookrightarrow X_0$ and the continuity of $\varphi: \setR \times X_0 \mapsto X_0$ (same argument as in \cite[Lemma 2.6]{hofer08_polyfolds}). The $\scale$-smoothness of $\varphi$ then follows by bootstrapping, since $\i \Delta = V: X^1 \mapsto X$ is $\scale^0$ linear.
	
	Finally, from \eqref{eq:flow_free_schrodinger} and \eqref{eq:dflow_free_schrodinger}, it is clear that \eqref{eq:sc_flow_equations} holds.
\end{example}

We end this subsection with a brief note on existence and uniqueness of $\scale$-smooth flows. In general, it is a hard question whether a given $\scale$-smooth vector field has an (even only locally defined) flow. In fact, most references are careful when it comes to general well\hyp{}posedness of Hamiltonian PDEs, either assuming it in some form \cite{kuksin00,abbondandolo14,ammari16} or deducing it for specific Hamiltonian PDEs and under specific assumptions ({\it e.g.}, \cite{bourgain93,ozawa07,murphy16,schoberg70,bona88,wang02,castelli10}). In \cite{crespo17}, we give a simple counter-example for general existence, and hint at the breaking down of uniqueness as well (except for very controlled examples).

\subsection{Hamiltonian Flows on Symplectic Scales} \label{sec:symplectic_scales_hamflows}
In this central part of the section, our aim is to define, in a meaningful way, what it means for an $\scale$\hyp{}smooth vector field and flow to be Hamiltonian. To do so, we need to introduce a new notion of smoothness for Hamiltonian functions. With this new notion, a smooth real\hyp{}valued function generates an $\scale$\hyp{}smooth vector field by means of a symplectic structure and corresponding symplectic gradient relation, similarly to the finite\hyp{}dimensional case. We derive a chain rule for this notion which is valid while pre\hyp{}composing with $\scale^\infty$\hyp{}symplectomorphisms, and which enables its usage with $\scale^\infty$\hyp{}manifolds in the following section. In the following, we denote the restriction of a Banach scale $X$ on $\setZ$ to $\setNz$ by $X_{\geq 0}$.

We start with a symplectic Banach scale $(X, \omega)$ on $\setZ$. As in the free Schrödinger equation, we can only expect Hamiltonian functions for Hamiltonian PDEs to be densely defined and, as such, we need to work with scale maps $h: X_{\geq 0}^1 \mapsto \setR$. To motivate the need of a new smoothness concept, analyse the usual $\omega$-gradient relation used pointwise to obtain the vector field $V_h$ from the Hamiltonian $h$
\begin{equation} \label{eq:sc_symplectic_gradient}
	-\D h = \omega(\cdot, V_h) \,.
\end{equation}
In the scale framework, we wish to obtain a map $(V_h)_m: X_{m+1} \mapsto X_m$ for each $m\geq 0$. Since $\omega$ pairs $X_{-m}$ with $X_{m}$, the derivative $\D_x h$ should be an element of $X_{-m}^*$ for each $x\in X_{m+1}$. Hence, the ``new derivative'' should induce a map $(\D h)_m: X_{m+1} \times X_{-m} \mapsto \setR, \, (x,\xi) \mapsto \D_x h \cdot \xi$ for each $m\geq 0$ which is linear in the second argument. 

In principle, given a Hamiltonian function $h: X_{\geq 0}^1 \mapsto \setR$, it would be possible to use the theory by Hofer of Section~\ref{sec:sc_calculus} to define an $\scale$\hyp{}derivative $(\D h)_m: X_{m+1} \times X_m \mapsto \setR$, $m\geq 0$, since the condition of an $\scale^1$ map in \thref{def:sc1}\ref{def:sc1_defn} does not need the map to be defined on the zeroth layer. It is not difficult to double\hyp{}check the proofs in \cite{hofer08_polyfolds,hofer10} and see that the theory carries over {\it mutatis mutandis} for these densely\hyp{}defined maps. Nevertheless, comparing this derivative with the desired form of the last paragraph, we see that test vectors are taken from $X_m$ instead of $X_{-m}$. Since the former space is smaller (remember that $m\geq 0$), the Hofer $\scale^1$ requirement is not strong enough to obtain a scale map $V_h: X_{\geq 0}^1 \mapsto X$. Indeed, the only case where the test spaces match is $m=0$, and with the original scale theory, relation \eqref{eq:sc_symplectic_gradient} only provides a vector field $V_h: X_1 \mapsto X_0$ without any scale structure {\it a priori}. In contrast, in our concept, we allow the smoothness of test vectors to decrease as the smoothness of the differentiation point increases, thereby obtaining a scale structure on $V_h$.

For densely\hyp{}defined maps, we shall refer to the marginally modified Hofer $\scale$\hyp{}smoothness concept as densely\hyp{}defined $\scale$\hyp{}smoothness, and to our alternative as {\it strongly} densely\hyp{}defined $\scale$\hyp{}smoothness\footnote{This concept is disjoint from the definition of a strong $\scale^k$ map in \cite[Remark~1.3]{hofer17}: the latter is simply a map which is $C^k$ on each layer.}. For the latter, ``densely\hyp{}defined'' will be frequently omitted from the terminology, seen that this is the only kind of strong $\scale$\hyp{}smoothness this paper deals with. For clarity, we reproduce the definition of densely\hyp{}defined $\scale$\hyp{}smoothness and, subsequently, we introduce the new smoothness concept.

\begin{definition}
Let $X$ be a Banach scale on $\setNz$, and let $U \subset X$ be open. An $\scale^0$ map $h: U^1 \mapsto \setR$ is said to be densely\hyp{}defined $\scale^1$ if there exists an $\scale^0$ map $\D h: U^1 \times X \mapsto \setR$ which is linear in the second argument and such that for all $x\in U_1$
	\begin{equation} \label{eq:frechet_condition_bis}
		\lim_{t\to 0} \frac{|h(x+t) - h(x) - \D h(x, t)|}{\|t\|_{X_1}} = 0
	\end{equation}
	as $t \to 0\in X_1$. We use the notation $\D_x h \eqdef \D h(x,\cdot) \in X_m^*$ for $x\in U_{m+1}$, $m\in \setNz$. The map $h$ is densely\hyp{}defined $\scale^{k+1}$, $k\in\setN\cup\{\infty\}$ if it is densely\hyp{}defined $\scale^1$ and $\D h: U^1 \times X \mapsto \setR$ is $\scale^{k}$ (in the usual sense).
\end{definition}

\begin{definition} \thlabel{def:strong_sc_smoothness}
	Let $X$ be a reflexive Banach scale on $\setZ$, $U \subset X_{\geq 0}$ open and $h: U^1 \mapsto \setR$ be an $\scale^0$ map.
	\begin{enumerate}
		\item The map $h$ is called strongly densely\hyp{}defined $\scale^1$, or simply strongly $\scale^1$, if there exists an $\scale^0$ map $\D h: U^1 \mapsto X^*$ such that for all $x\in U_1$
		\begin{equation} \label{eq:frechet_condition_hamiltonian}
			\frac{|h(x+t) - h(x) - \D h(x)\cdot t|}{\|t\|_{X_1}} \to 0
		\end{equation}
		as $t\to 0$ in $X_1$. We use the notation $D_x h \eqdef \D h(x) \in X_{-m}^*$ for $x\in U_{m+1}$, $m\in\setNz$.
		\item The map $h$ is called strongly (densely\hyp{}defined) $\scale^{k+1}$ if it is strongly $\scale^1$ and $\D h: U^1 \mapsto X^*$ is $\scale^k$ in the original Hofer sense, $k\in\setN\cup\{\infty\}$.
	\end{enumerate}
\end{definition}
\begin{remark} \thlabel{rmk:strong_sc_smoothness}
	\begin{enumerate}
		\item \label{rmk:strong_sc_continuity} The reader might notice that for the definition of strongly $\scale^1$ maps, instead of requiring $\D h$ as above to be $\scale^0$, it would be more natural and compatible with the Hofer $\scale^1$ condition to require each map $U_{m+1} \times X_{-m} \mapsto \setR, (x,\xi) \mapsto \D_x h\cdot\xi$ to be continuous, $m\geq 0$. This weaker condition would suffice to prove the Frèchet condition \eqref{eq:frechet_condition_hamiltonian} in a chain rule scenario, but not the continuity of the derivative of the composed map (cf.~\thref{rmk:sc1minus_not_enough}).				
		\item If $h: U^1 \mapsto \setR$ is strongly $\scale^1$, then the one\hyp{}layer map $h: U_1 \mapsto \setR$ is $C^1$, since the inclusion $X_0^* \subset X_1^*$ is continuous. Also, by using the bonding (inclusion) maps $X_m \hookrightarrow X_{-m}$ for $m\geq 1$, the derivative $\D h: U^1 \mapsto X^*$ of a strongly $\scale^{1}$ map $h$ induces an $\scale^0$ map $U^1 \times X \mapsto \setR, \, (x,\xi) \mapsto \D_x h\cdot\xi$. Taking \eqref{eq:frechet_condition_hamiltonian} into account, we see that strongly $\scale^{1}$ maps $h$ are densely\hyp{}defined $\scale^1$. \thref{prop:relations_smoothness} expands on the relations between different smoothness concepts.
	\end{enumerate}
\end{remark}

In the same way as in Hofer scale calculus, one can prove that the derivative of a strongly $\scale^1$ map is unique, and that for a strongly $\scale^k$ map $h: U^1 \mapsto \setR$ and $V\subset U$ open, $h\restr{V_1}: V^1 \mapsto \setR$ is still strongly $\scale^k$ with $\D(h\restr{V_1}) = (\D h)\restr{V_1}: V^1 \mapsto X^*$. From this, one proves locality of the strong $\scale^k$ conditions, meaning that $h: U^1 \mapsto \setR$ is strongly $\scale^k$ if and only if for each $x\in U_1$, there exists an open neighbourhood $V(x) \subset U$ such that $h\restr{V_1}: V^1 \mapsto Y$ is strongly $\scale^k$.

As announced in the motivation of strong $\scale$\hyp{}smoothness, if $(X,\omega)$ is a symplectic Banach scale on $\setZ$, a strongly $\scale^{1}$ map $h:  X_{\geq 0}^1 \mapsto \setR$ induces an $\scale^0$ vector field $V_h: X_{\geq 0}^1 \mapsto X$ which is uniquely defined by the $\omega$-gradient relation \eqref{eq:sc_symplectic_gradient}, where the derivative $\D$ is in the strong sense. Since $\iota_\omega$ is an isomorphism of scales, $V_h$ is $\scale^k$ if and only if $h$ is strongly $\scale^{k+1}$, $k\in\setNz\cup\{\infty\}$. This leads to the following definition we have worked towards.
\begin{definition}
	Let $(X,\omega)$ be a symplectic Banach scale on $\setZ$. An $\scale$\hyp{}smooth vector field $V: X_{\geq 0}^1 \mapsto X$ is said to be Hamiltonian if there exists a strongly $\scale$\hyp{}smooth map $h: X_{\geq 0}^1 \mapsto \setR$ such that
	\begin{equation} \label{eq:sc_hamvect_symplectic_gradient}
		-\D h = \omega(\cdot, V)
	\end{equation}
	holds pointwise. If a Hamiltonian $\scale$\hyp{}smooth vector field $V$ has a global flow $\varphi: \setR \times X_{\geq 0} \mapsto X_{\geq 0}$, then the flow $\varphi$ is said to be Hamiltonian.
\end{definition}
\begin{example} \thlabel{ex:ham_free_schrodinger}
	For the Banach scale $X_s = W^{2s,2}(S^1, \setC)$, $s\in\setZ$, with its standard symplectic structure\footnote{recall that $X$ is seen as a real scale, and the inner product should be interpreted as the \emph{real\hyp{}valued} inner product of $X_0$.} $\omega = \langle \i \cdot, \cdot \rangle_0$, the vector field $V = \i \Delta$ and flow $\varphi(t,\cdot) = \e^{\i t\Delta}$ of \thref{ex:flow_free_schrodinger} are Hamiltonian. Indeed, consider $h: X_{\geq 0}^1 \mapsto \setR$ given by
	\begin{equation}
		h(u) = \frac{\|u_x\|_{0}^2}{2} = \frac{1}{2} \int_{S^1} |u_x(a)|^2 \,\d a \,,
	\end{equation}
	where $(\cdot)_x: X^1 \mapsto X$ is the weak differentiation operator (Fourier multiplier $[n \mapsto \i n] \in \setC^\setZ$). The Frèchet condition \eqref{eq:frechet_condition_hamiltonian} is satisfied with $\D_u h \cdot \xi = \langle u_x, \xi_x \rangle_0 = \int_{S^1} u_x(a) \cdot \xi_x(a)\, \d a$ for $u,\xi \in X_1$, and by integration by parts, this map extends to an $\scale^0$ map $\D h: X^1 \mapsto X^*, \, u \mapsto -\langle u_{xx}, \cdot \rangle_0$. Since $\D h$ happens to be linear, we conclude that $h$ is strongly $\scale$\hyp{}smooth. For $u\in X_{m+1}$, $m \geq 0$, we then have $\omega(\cdot, \i \Delta u) = \langle \cdot, u_{xx} \rangle_0 = -\D_u h$.
\end{example}

The following proposition clarifies the relationships between the several smoothness concepts used so far. Note in the proposition that a scale map $h: U \subset X_{\geq 0} \mapsto \setR$ which is $C^k$ on each layer satisfies $(1)$, and that a map which satisfies $(4)$ is $C^k$ on each layer as a map $h: U^k \mapsto \setR$.
\begin{proposition} \thlabel{prop:relations_smoothness}
	Let $X$ be a reflexive Banach scale on $\setZ$, $U \subset X_{\geq 0}$ be open, and $h: U \mapsto \setR$ be an $\scale^0$ map. We have the following implication diagram, where $A \Longrightarrow B$ means ``$A$ implies $B$'' and $A \Longdoesnotimply B$ means ``$A$ does in general not imply $B$'':
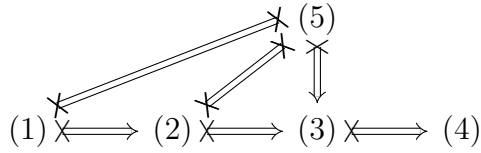
	\[ \begin{tikzcd}[arrows=Rightarrow]
		{} & {} & (5) \ar[notsrc]{d} & {} \\
		(1) \ar[notsrc]{r}\ar[equal,notsrc,notdst,shift left=5pt]{rru} & (2)\ar[notsrc]{r}\ar[equal,notsrc,notdst]{ru} & (3) \ar[notsrc]{r} & (4)
	\end{tikzcd} \]
	\noindent where, for $k\in\setN$, we label: \\
	(1) $h: U_m \mapsto \setR$ is $C^{m+1}$ for $m\in\{0,1,\ldots,k-1\}$; \\
	(2) $h: U \mapsto \setR$ is $\scale^k$; \\
	(3) $h: U^1 \mapsto \setR$ is densely\hyp{}defined $\scale^k$; \\
	(4) $h: U_m \mapsto \setR$ is $C^m$ for $m\in\{1,2,\ldots,k\}$; \\
	(5) $h: U^1 \mapsto \setR$ is strongly densely\hyp{}defined $\scale^k$.
\end{proposition}
\begin{proof}
	$(1) \Rightarrow (2)$ and $(3) \Rightarrow (4)$ are proved in \cite[Proposition 2.4]{hofer10} and \cite[Proposition 2.3]{hofer10}, respectively. $(2) \Rightarrow (3)$ is direct from the definition. To prove $(5) \Rightarrow (3)$, note that we have an $\scale^0$ bilinear evaluation map $\ev: X^* \times X_{\geq 0} \mapsto \setR, \, (T \in X_{-m}^*, x \in X_m) \mapsto T\restr{X_m}(x)$, $m\geq 0$. Consequently, the $\scale^{k-1}$ derivative $\D h: U^1 \mapsto X^*$ induces an $\scale^{k-1}$ derivative $\ev \circ(\D h \times \id_{X_{\geq 0}}): U^1 \times X \mapsto \setR$ which satisfies the Frèchet condition by hypothesis.
	
	To prove the counter\hyp{}implications, we provide four counter\hyp{}examples. Firstly, with $X_s = l^{2,s}_\setR$, $s\in\setZ$, define $h: X_{\geq 0} \mapsto \setR, \, x \mapsto \langle x, y \rangle_0$ for some $y \in X_0 \setminus X_{1/2}$. Here, we use rational indices to the scale in an informal manner: the expression \eqref{eq:l2s_space} for $l^{2,s}$ is actually defined for every $s\in\setR$. With this definition, $h: X_m \mapsto \setR$ is $C^\infty$ for all $m\in\setNz$, and from $\d h(x) \cdot \xi = \langle \xi, y \rangle_0$, $x,\xi \in X_0$, it is easy to see that $\d h(x): X_0 \subset X_{-1} \mapsto \setR$ is not continuous for $x\in X_2$. This proves $(1) \Doesnotimply (5)$.
	
	The second counter\hyp{}example is similar and proves $(4) \Doesnotimply (3)$. Define $h: X_{\geq 0}^1 \mapsto \setR, \, x \mapsto \frac{1}{2} \|x\|_1^2$ for the same scale $X$. We have $h: X_m \mapsto \setR$ is $C^\infty$ for all $m\geq 1$ but the derivative $\d h(x) = \langle x, \cdot \rangle_1: X_1 \mapsto \setR$ cannot be extended to a continuous linear map $X_0 \mapsto \setR$ if we take $x\in X_1\setminus X_{3/2}$.
	
	Thirdly, to prove $(5) \Doesnotimply (2)$, let now $X_s = W^{2s,2}(S^1, \setC)$, $s\in\setZ$. The Hamiltonian for the free Schrödinger equation in \thref{ex:ham_free_schrodinger} is densely\hyp{}defined $\scale$\hyp{}smooth but cannot be extended to a map $h: X_{\geq 0} \mapsto \setR$ satisfying $(2)$, since it is not even continuous with respect to the topology of $X_0$.
	
	Last but not least, we prove $(2) \Doesnotimply (1)$. For the scale $X$ of the last paragraph, let $h: \setR \times X_{\geq 0} \mapsto \setR, \, (t,u) \mapsto \langle \e^{\i t \Delta} u, v\rangle_{0,\setR}$, where $v \in X_0\setminus X_{1/2}$. This map is $\scale$\hyp{}smooth but its zeroth layer $h_0: \setR \times X_0 \mapsto \setR$ is not $C^1$. Indeed, if that was the case, the $\scale$\hyp{}derivative $\D h: \setR \times X_{\geq 0}^1 \times \setR \times X_{\geq 0} \mapsto \setR$ would be such that $(\D h)_0(0,\cdot,1,0): X_1 \subset X_0 \mapsto \setR, \, u \mapsto \langle \i\Delta u, v \rangle_0$ is continuous, which is not the case. The remaining counter\hyp{}implications are a consequence of these four.
\end{proof}

\thref{def:strong_sc_smoothness} introduced the concept of strong $\scale$\hyp{}smoothness on a reflexive Banach scale $X$ on $\setZ$. To generalize this concept to $\scale^\infty$\hyp{}manifolds later on, we need it to be invariant under $\scale$\hyp{}smooth coordinate changes. For $U \subset X_{\geq 0}$ and $V \subset Y_{\geq 0}$ open, a strongly $\scale^k$ map $h: V^1 \mapsto \setR$, $k\in \setN\cup\{\infty\}$, and an $\scale^\infty$\hyp{}diffeomorphism $f: U \isomapsto V$, it is then the question whether the composition $h\circ f: U^1 \mapsto \setR$ is also strongly $\scale^{k}$. To answer this question positively we need, for each $x\in X_{m+1}$, $m\geq 0$, to map $\D_{f(x)} h \in Y_{-m}^*$ to an element $\D_x(h\circ f) \in X_{-m}^*$ to be defined. If $\D_x f$ existed as a continuous linear map $X_{-m} \mapsto Y_{-m}$, we could take its adjoint for this, obtaining $\D_x(h\circ f) = \D_{f(x)} h \circ \D_x f$ as usual. Nevertheless, $f$ only defines a scale morphism $\D_x f: X\restr{m+1} \mapsto Y\restr{m+1}$ for the non\hyp{}negative indices $m+1 = \{0,1,\ldots, m\}$.

To solve the problem raised above, we need to extend $\D_x f$ to a scale morphism on $\{-m, -m+1, \ldots, m\}$. If we assume that we have symplectic structures $\omega$ and $\eta$ on $X$ and $Y$, respectively, then, by the discussion at the end Section \ref{sec:banach_scales}, we can use these and $(\D_x f)^{-1}: Y\restr{m+1} \mapsto X\restr{m+1}$ to obtain a scale morphism $((\D_x f)^{-1})^{\eta,\omega}: X\restr{\{-m, -m+1, \ldots, 0\}} \mapsto Y\restr{\{-m, -m+1, \ldots, 0\}}$. In order that the morphisms $\D_x f$ and $((\D_x f)^{-1})^{\eta,\omega}$ glue together to a morphism on $\{-m, -m+1, \ldots, m\}$, we need them to coincide on the zeroth layer, {\it i.e.}, $(\D_x f)_0 = ((\D_x f)^{-1})^{\eta,\omega}_0: X_0 \mapsto Y_0$. This condition precisely means that $\D_x f$ should be a linear symplectomorphism of scales for all $x\in X_{m+1}$ and $m\in\setNz$. Of course, it is enough to require this condition for $m=0$. The following definition and proposition solidify this discussion.

\begin{definition}
	Let $(X,\omega)$ and $(Y,\eta)$ be symplectic Banach scales on $\setZ$ and $U\subset X_{\geq 0}$ be open. An $\scale$\hyp{}smooth map $f: U \subset X_{\geq 0} \mapsto Y_{\geq 0}$ is called symplectic whenever $\D_x f: X_0 \mapsto Y_0$ is symplectic for all $x\in U_1$, that is,
	\begin{equation}
		\eta(\D_x f\cdot v, \D_x f\cdot w) = \omega(v, w)
	\end{equation}
	for all $v,w \in X_0$ and $x\in U_1$ (or equivalently, $v,w \in X_\infty$ and $x \in U_\infty$). It is an $\scale^\infty$\hyp{}symplectomorphism if, in addition, there exists $V\subset Y_{\geq 0}$ open with $f(U) = V$ and $f: U \mapsto V$ is an $\scale^\infty$\hyp{}diffeomorphism.
\end{definition}

\begin{proposition} \thlabel{prop:chain_rule_strong_sck}
	(Chain rule for strong $\scale$\hyp{}maps) Let $(X, \omega)$ and $(Y, \eta)$ be symplectic Banach scales on $\setZ$, $U\subset X_{\geq 0}$ and $V\subset Y_{\geq 0}$ be open, $h: V^1 \mapsto \setR$ be an $\scale^0$ map and $f: U \isomapsto V$ be an $\scale^\infty$\hyp{}symplectomorphism. Then, for $k\in \setN \cup \{\infty\}$, $h$ is strongly $\scale^{k}$ if and only if $h \circ f: U^1 \mapsto \setR$ is so. If this is the case, the chain rule
	\begin{equation} \label{eq:chain_rule_strong_sc}
		\D_x(h\circ f) = \D_{f(x)} h \circ ((\D_x f)^{-1})^{\eta, \omega}
	\end{equation}
	holds for all $x \in U_{m+1}$, $m\in\setNz$.
\end{proposition}
\begin{proof}
	Clearly, it suffices to prove the ``only if'' part, since $f^{-1}$ is also an $\scale^\infty$\hyp{}symplectomorphism. Starting with the regularity of the candidate derivative \eqref{eq:chain_rule_strong_sc}, apply the adjoint construction $((\cdot)^{\eta, \omega})^*$ pointwise to $\D (f^{-1}): V^1 \times Y \mapsto X$ to obtain a scale map $(\D (f^{-1})^{\eta, \omega})^*: V^1 \times Y^* \mapsto X^*$ given by the diagram
	\begin{equation} \label{diag:sympl_adj_dfinv}
		\begin{tikzcd}[column sep=large]
			V^1 \times Y \ar{r}{\D (f^{-1})} \ar{d}{\sim}[swap]{\id_{V^1} \times \iota_\eta}	& X \ar{d}{\iota_\omega}[swap]{\sim} \\
			V^1 \times Y^* \ar[dashed]{r}{(\D (f^{-1})^{\eta, \omega})^*}								& X^*
		\end{tikzcd} \,.
	\end{equation}
	Since $f$ and the vertical maps in the diagram are $\scale^\infty$\hyp{}diffeomorphisms, it follows that $(\D (f^{-1})^{\eta, \omega})^*$ is $\scale$\hyp{}smooth. By \eqref{eq:chain_rule_strong_sc} and the chain rule for $f$ of \thref{prop:chain_scsmooth}, we have $\D (h\circ f) = (\D (f^{-1})^{\eta, \omega})^* \circ (\id_{V^1} \times \D h) \circ \diag_{V^1} \circ f\restr{U_1}$. Consequently, $\D (h\circ f)$ is as smooth as $\D h$ is.

	To prove the Frèchet condition \eqref{eq:frechet_condition_hamiltonian}, we first note that since $h: V_1 \mapsto \setR$ is $C^1$, the fundamental theorem of calculus together with the hypothesis that $f$ is symplectic gives for $x\in U_1$ and $t\in X_1$ small 
	\begin{flalign}
		& h(f(x+t)) - h(f(x)) - \D h (f(x)) \circ (\D f(x)^{-1})^{\eta, \omega} \cdot t \nonumber \\
		& \; = \int_0^1 \D h(a f(x+t) + (1-a) f(x)) \cdot (f(x+t)-f(x))\, \d a \\
		& \;\quad - \D h (f(x)) \circ \D f(x) \cdot t \nonumber \\
		& \; = \int_0^1 \D h(a f(x+t) + (1-a) f(x)) \cdot (f(x+t)-f(x) - \D f(x)\cdot t) \,\d a  \label{eq:chain_rule_strong_sc_twoterms} \\
		& \;\quad + \int_0^1 [\D h(a f(x+t) + (1-a) f(x)) - \D h (f(x))] \circ \D f(x) \cdot t\, \d a \,. \nonumber
	\end{flalign}
	The remainder of the proof is similar to the original proof for $\scale^1$ maps \cite[Theorem~2.16]{hofer06}. The integrand of first term in \eqref{eq:chain_rule_strong_sc_twoterms} divided by $\|t\|_{X_1}$ converges to 0 uniformly in $a \in [0,1]$ as $t\to 0$ in $X_1$ due to the $\scale$\hyp{}differentiability of $f$ and the continuity of $\D h: V_1 \mapsto Y_0^*$. In turn, the second integrand term divided by $\|t\|_{X_1}$ converges uniformly to 0 due to the compactness of $\{\D_x f \cdot \frac{t}{\|t\|_1}: t \in X_1\setminus\{0\}\} \subset Y_0$ and again the continuity of $\D h$.
\end{proof}
\begin{remark} \thlabel{rmk:sc1minus_not_enough}
	As hinted in \thref{rmk:strong_sc_smoothness}\ref{rmk:strong_sc_continuity}, we could have loosened the $\scale^0$ continuity of $\D h$ to the requirement that $(\D h)_m: V_{m+1} \times Y_{-m} \mapsto \setR, (x,\xi) \mapsto \D_x h\cdot\xi$ be continuous for each $m\geq 0$, and the last paragraph of this proof would still hold as in the original proof of Hofer. Nevertheless, this weaker requirement would need the continuity of $\D (f^{-1})_m: V_{m+1} \mapsto \B(Y_m, X_m)$ with respect to a stronger topology on $\B(Y_m, X_m)$ than the compact\hyp{}open topology to prove the continuity of $(\D(h\circ f))_m: U_{m+1} \times X_{-m} \mapsto \setR$. The issue here is that for this alternative definition, we would need to endow each space in the Banach scales $X^*$ and $Y^*$ with the compact\hyp{}open topology, with the consequence that the vertical maps in \eqref{diag:sympl_adj_dfinv} would not be levelwise homeomorphisms anymore.
\end{remark}

\subsection{Hamiltonian Flows on Symplectic Scale Manifolds} \label{sec:sc_mfld_flows}
In this subsection, we generalize the concepts of Section~\ref{sec:symplectic_scales_hamflows} to the case of an $\scale^\infty$\hyp{}manifold. To accomplish this task, we need to introduce new structures on the manifolds, such as an extension of the tangent scales to negative indices and cotangent bundles. It turns out that the crucial requirement to enable this is that the transition maps are symplectic. This condition gives rise to the concept of a symplectic $\scale^\infty$\hyp{}manifold, where we can define the new objects appealing to the local model by means of a coordinate chart. Due to the assumption on the transition maps, the result is independent of the coordinate chart used to define the structure. Once the desired structures are formed, we obtain an elegant, direct and natural generalization of strong $\scale$\hyp{}smooth maps, Hamiltonian vector fields and Hamiltonian flows for the case of $\scale^\infty$\hyp{}manifolds.

We begin directly by defining symplectic $\scale$\hyp{}smooth manifolds and presenting the relevant example. Again, for a Banach scale $X$ on $\setZ$, $X_{\geq 0} = X\restr\setNz$.
\begin{definition}
	Let $(X,\omega)$ be a symplectic Banach scale on $\setZ$ and let $M$ be an $\scale^\infty$\hyp{}manifold locally modeled on $X_{\geq 0}$.
	\begin{enumerate}
		\item Two coordinate charts $(U,\phi)$ and $(U', \psi)$ of $M$ are said to be symplectically compatible if the transition map $\psi\circ\phi^{-1}: \phi(U\cap U') \mapsto \psi(U\cap U')$ is an $\scale^\infty$\hyp{}symplectomorphism of $(X,\omega)$.
		\item A symplectic atlas for $M$ is an atlas $\mathcal{A} = \{(U_a,\phi_a)\}_{a\in A}$ for the $\scale$-smooth structure of $M$ such that $(U_a,\phi_a)$ and $(U_b,\phi_b)$ are symplectically compatible for all $a,b\in A$.
		\item If a symplectic atlas $\mathcal{A}$ for $M$ exists, it is contained in a unique maximal symplectic atlas $\bar{\mathcal{A}}$. The pair $(M,\bar{\mathcal{A}})$ is then said to be a symplectic $\scale^\infty$\hyp{}manifold locally modeled on $(X,\omega)$, and $\bar{\mathcal{A}}$ is its symplectic $\scale$\hyp{}smooth structure. Usually, the latter is suppressed from notation.
	\end{enumerate}
\end{definition}
Unless otherwise stated, we always take coordinate charts of a symplectic $\scale^\infty$\hyp{}manifold from its symplectic $\scale$\hyp{}smooth structure, whence the transition maps are always assumed to be symplectic.

\begin{example} \thlabel{ex:symplectic_projective_hilbert}
	The projectivization of a complex separable Hilbert space $X$ of \thref{ex:projective_hilbert} is a symplectic $\scale^\infty$\hyp{}manifold if we endow the induced scale $\scf{X}$ with its standard symplectic structure. Again $X = l_\setC^2$ without loss of generality, since the isometric isomorphism $\scf{X} \isomapsto \scf{l_\setC^2}$ is symplectic (by definition). With $U_a$ as in \thref{ex:projective_hilbert}, $a\in\setZ$, and being $B = \{x \in X_0: \|x\|_0 < 1\}$ the unit ball of $X=X_0$, we can define a symplectic atlas $\{(U_a,\psi_a)\}_{a\in\setZ}$ with $\psi_a: U_a \isomapsto B \subset X_0$ given by
	\begin{equation} \label{eq:darboux_charts_projective_hilbert}
		\psi_a([x])_n = \frac{|x_a|}{x_a \|x\|_0} \cdot
		\begin{cases}
			x_n			&\txt{if}~n<a \\
			x_{n+1}		&\txt{if}~n\geq a \,.
		\end{cases}
	\end{equation}
\end{example}

For a symplectic $\scale^\infty$\hyp{}manifold $M$ modeled on $(X,\omega)$ and $p\in M_{m+1}$, $m\in \setNz\cup\{\infty\}$, choose a coordinate chart $(U, \phi)$ of $M$ around $p$. By following the procedure of \thref{rmk:symplectic_scale_extension} with the induced isomorphism of scales $\D_p\phi: \Tg_p M \isomapsto X\restr{m+1}$, we can extend the tangent scale $\Tg_p M$ to a scale on $\{-m, -m+1, \ldots, m\}$ (on $\setZ$ if $m=\infty$). Recall that this is done by declaring $(\Tg_p M)_s \eqdef (\Tg_p M)_{-s}^*$ and $(\D_p\phi)_s \eqdef ((\D_p\phi)_{-s}^{~-1})^*: (\Tg_p M)_s \isomapsto X_{-s}^* \stackrel{\iota_\omega}{\mapsfrom} X_s$ for $-m\leq s <0$, subsequently pulling back the bonding maps of $X\restr{\{-m,-m+1,\ldots,0\}}$. Due to the invariance of the transition maps of $M$ under $\omega$, this is a well\hyp{}defined procedure which is independent of the chosen chart $(U, \phi)$. Furthermore, since $X$ is a reflexive scale, $\Tg_p M$ is reflexive as well, whence it admits a dual scale $\Tg^*_p M$. It is then clear that we can pull back the symplectic scale structure of $X$ using the diagram
	\begin{equation} \label{eq:symplectic_omegap} \begin{tikzcd}
		\Tg_p M \ar[dashed]{r}{\iota_{\omega_p}}[swap]{\sim} \ar{d}{\sim}[swap]{\D_p \phi} & \Tg_p^* M \ar{d}{((\D_p \phi)^{-1})^*}[swap]{\sim} \\
		X \ar{r}{\iota_\omega}[swap]{\sim} & X^*
	\end{tikzcd} \,, \end{equation}
where $X$ in the lower row are restricted to $\{-m,-m+1,\ldots,m\}$. Indeed, from this diagram we recover skew\hyp{}symmetric bilinear maps $\omega_p \eqdef (v,w) \mapsto \iota_{\omega_p}(w)\cdot v: (\Tg_p M)_{-s} \times (\Tg_p M)_{s} \mapsto \setR$, $s\in \{-m,-m+1,\ldots,m\}$, 
which make $(\Tg_p M, \omega_p)$ into a symplectic Banach scale.

A structure which follows from extended tangent scales is the cotangent bundle. Once we define this bundle, we rewrite the newly constructed isomorphism in \eqref{eq:symplectic_omegap} globally. Also, we introduce $\scale^k$ sections on the cotangent bundle in the usual manner.
\begin{proposition}
	Let $M$ be a symplectic $\scale^\infty$\hyp{}manifold locally modeled on $(X,\omega)$. Then, the cotangent bundle
	\begin{equation}
		\Tg^* M \eqdef \cup_{p\in M_1} \{p\}\times (\Tg_p M)_0^*
	\end{equation}
	is an $\scale^\infty$\hyp{}manifold locally modeled on $X_{\geq 0}^1 \times X^*$. Its induced filtration is $(\Tg^* M)_m = \cup_{p\in M_{m+1}} \{p\} \times (\Tg_p M)^*_{-m}$, $m\in\setNz$, where $(\Tg_p M)^*_{-m} \subset (\Tg_p M)^*_0$ via the adjoint of the bonding map $(\Tg_p M)_0 \hookrightarrow (\Tg_p M)_{-m}$. Also, the bundle projection $\pi_{\Tg^* M}: \Tg^* M \mapsto M^1, \, (p,\alpha) \mapsto p$ is a surjective $\scale$\hyp{}smooth submersion and the fiber $\pi_{\Tg^* M}^{-1}(p)$ over $p\in M_{m+1}$ is $\{p\} \times (\Tg_p M)^*_0 \iso (\Tg_p M)^*_0$ with the scale structure of $(\Tg^*_p M)_{\geq 0}$: $(\Tg_p M)^*_{-m} \subset (\Tg_p M)^*_{-m+1} \subset \ldots \subset (\Tg_p M)^*_0$.
\end{proposition}
\begin{proof}
	The methodology is similar to the construction of the tangent bundle. Letting $\pi: \Tg^* M \mapsto M_1$ be solely a map of sets in the first place, we define $\Tg^* U \eqdef \pi_{\Tg^* M}^{-1}(U_1)$ for $U\subset M$ open. A coordinate chart $(U, \phi)$ of $M$ induces a bijection $\Tg^* \phi: \Tg^* U \mapsto V_1 \times X_0^*,\, (p,\alpha) \mapsto (\phi(p), ((\D_p\phi)_0^{-1})^* \cdot \alpha)$. If $(U', \psi)$ is an additional coordinate chart we have, on the one hand, $(\D_x(\psi\phi^{-1})_0^{\omega,\omega})^* = (\D_x(\psi\phi^{-1})_0^{-1})^*: X_0^* \isomapsto X_0^*$ for all $x\in \phi(U\cap U')_1$ since $\psi\circ\phi^{-1}: \phi(U\cap U') \isomapsto \psi(U\cap U')$ is an $\scale^\infty$\hyp{}symplectomorphism. On the other hand, using a diagram similar to \eqref{diag:sympl_adj_dfinv}, we see that $(\D(\psi\phi^{-1})^{\omega,\omega})^*: \phi(U\cap U')^1 \times X^* \mapsto X^*$ is $\scale$\hyp{}smooth, where the operations are taken pointwise. Together, these conditions imply that the induced transition map $\Tg^* \psi \circ (\Tg^* \phi)^{-1}: \phi(U\cap U')^1 \times X^* \isomapsto \psi(U\cap U')^1 \times X^*$ preserves scales and is $\scale$\hyp{}smooth. The remaining statements are easily verified.
\end{proof}

\begin{remark} \thlabel{rem:symplectic_sc_bundleiso}
	The cotangent bundle $\Tg^* M$ allows us to globally rewrite the isomorphism $\iota_{\omega_p}$ in \eqref{eq:symplectic_omegap} by collecting $p\in M_{m+1}$, $m\in\setNz$. Indeed, we can define the $\scale^\infty$\hyp{}diffeomorphism $\iota_\omega: \Tg M \isomapsto \Tg^* M, \, (p, v) \mapsto (p, \iota_{\omega_p}(v))$. This diffeomorphism maps fibers of $\Tg M$ to fibers of $\Tg^* M$ and is linear on each fiber, or in other words, it is an isomorphism of $\scale$\hyp{}smooth vector bundles.
\end{remark}

\begin{definition} \thlabel{def:sections_sc_cotg_bundle}
	An $\scale^k$ section of the cotangent bundle of a symplectic $\scale^\infty$\hyp{}manifold $M$, $k\in\setNz\cup\{\infty\}$, is an $\scale^k$ map $s: M^1 \mapsto \Tg^* M$ with $\pi_{\Tg^* M}\circ s = \id_{M^1}$.
\end{definition}

The final coordinate\hyp{}free structure that we introduce on symplectic $\scale^\infty$\hyp{}manifolds are strongly densely\hyp{}defined $\scale^k$ maps. This is a simple generalization of the concept in Section~\ref{sec:symplectic_scales_hamflows} and its coordinate independence is a direct consequence of \thref{prop:chain_rule_strong_sck} and of the locality of the strong $\scale^k$ conditions. The proof of well\hyp{}definedness is a simple manipulation of the concepts introduced so far and will be omitted.
\begin{definition}
	Let $M$ be a symplectic $\scale^\infty$\hyp{}manifold locally modeled on a symplectic Banach scale $(X,\omega)$ on $\setZ$.
	\begin{enumerate}
		\item An $\scale^0$ map $h: M^1 \mapsto \setR$ is said to be strongly (densely\hyp{}defined) $\scale^k$, $k\in\setN\cup\{\infty\}$, if for all $p\in M_1$ there exists a coordinate chart $\phi: U(p) \isomapsto V \subset X_0$ of $M$ such that $h\circ \phi^{-1}: V^1 \mapsto \setR$ is strongly densely\hyp{}defined $\scale^k$.
		\item For a strongly $\scale^k$ map $h: M^1 \mapsto \setR$, we define at each $p\in M_{m+1}$, $m\in\setNz$, the $m\th$ level derivative of $h$ at $p$ to be $(\D_p h)_m \eqdef \D_{\phi(p)}(h\circ\phi^{-1}) \circ (\D_p \phi)_{-m} \in (\Tg_p M)_{-m}^*$.
		\item Varying $p$ above, we obtain the derivative of $h$: an $\scale^{k-1}$ section of the cotangent bundle $\D h: M^1 \mapsto \Tg^* M, \, M_{m+1} \owns p \mapsto (p, (\D_p h)_m) \in (\Tg^* M)_m$, $m\in\setNz$.
	\end{enumerate}
\end{definition}

The technical work carried out above drastically facilitates the task of generalizing Hamiltonian vector fields and flows to symplectic $\scale^\infty$\hyp{}manifolds. Indeed, this is now a question of seamlessly combining the toolkit developed in this document. First, for an $\scale^\infty$\hyp{}manifold $M$ (not compulsorily symplectic), we define an autonomous $\scale$\hyp{}smooth vector field $V: M^1 \mapsto \Tg M$ simply to be an $\scale$\hyp{}smooth section of the tangent bundle as in \thref{def:sections_sc_tg_bundle}.

As seen at the end of Section~\ref{sec:sc_manifolds}, $\scale$\hyp{}smooth maps $\varphi: \setR \times M \mapsto M$ define a partial derivative $\frac{\partial \varphi}{\partial t}(t,p) \in (\Tg_{\varphi(t,p)} M)_m$ for $t\in\setR$ and $p\in M_{m+1}$, $m\in\setNz$. Varying $t$ and $p$, we obtain an $\scale$\hyp{}smooth map
\begin{equation}
	\frac{\partial \varphi}{\partial t}: \setR \times M^1 \mapsto \Tg M, \, (t,p) \mapsto \Big(\varphi(t,p), \frac{\partial \varphi}{\partial t}(t,p) \Big) \,,
\end{equation}
which is locally the diagonal product of $\varphi$ in local coordinates and its partial derivative as a map of scales, as used in Section~\ref{sec:banach_scales_flows}. An $\scale$\hyp{}smooth vector field $V: M^1 \mapsto \Tg M$ is then said to have a global flow if there exists an $\scale$\hyp{}smooth map $\varphi: \setR \times M \mapsto M$ such that
\begin{align}
	\frac{\partial \varphi}{\partial t}(t,p) &= V \circ \varphi(t,p) & \varphi(0,p) &= p
\end{align}
for all $t\in\setR$ and $p\in M_{m+1}$, $m\in\setNz$. As in Section~\ref{sec:banach_scales_flows}, it is sufficient to check this condition for $m=0$, and the vector field can be recovered from $V = \frac{\partial \varphi}{\partial t}(0,\cdot)$.

For a symplectic $\scale^\infty$\hyp{}manifold $M$, a strongly $\scale$\hyp{}smooth map $h: M^1 \mapsto \setR$ gives rise to an $\scale$\hyp{}smooth vector field $V_h: M^1 \mapsto \Tg M$ uniquely defined by the relation
\begin{equation}
	-\D h = \iota_\omega \circ V_h : M^1 \mapsto \Tg^* M \,,
\end{equation}
since $\iota_\omega: \Tg M \isomapsto \Tg^* M$ is an $\scale$\hyp{}smooth vector bundle isomorphism (as described in \thref{rem:symplectic_sc_bundleiso}). Fiberwise, this equation reads $-(\D_p h)_m = \omega_p(\cdot, V_h(p))$ for $p\in M_{m+1}$, $m\in\setNz$. The definition of a Hamiltonian vector field and flow is now apparent.
\begin{definition}
	Let $M$ be a symplectic $\scale^\infty$\hyp{}manifold and let $\iota_\omega: \Tg M \isomapsto \Tg^* M$ be the induced $\scale$\hyp{}smooth vector bundle isomorphism. An $\scale$\hyp{}smooth vector field $V: M^1 \mapsto \Tg M$ is said to be Hamiltonian if there exists a strongly $\scale$\hyp{}smooth map $h: M^1 \mapsto \setR$ such that
	\begin{equation}
		-\D h = \iota_\omega \circ V : M^1 \mapsto \Tg^* M \,.
	\end{equation}
	If a Hamiltonian vector field has a global flow $\varphi: \setR \times M \mapsto M$, the flow is said to be Hamiltonian.
\end{definition}

\begin{example} \thlabel{ex:reduced_free_schrodinger}
	The flow of the free Schrödinger equation in \thref{ex:flow_free_schrodinger}, on the Banach scale $X = \{W^{2s,2}(S^1, \setC)\}_{s\in\setZ}$ (restricted to $\setNz$), descends to the projectivization $M = \projvect(X_0)$, and the descended flow map $\bar\varphi: \setR \times M \mapsto M$ is easily seen to be $\scale$\hyp{}smooth. We claim that this flow is Hamiltonian. In the first place, with the coordinates of \thref{ex:symplectic_projective_hilbert}, the vector field $\bar V = \frac{\partial \bar \varphi}{\partial t}(0,\cdot)$ is given by ${\bar V}_a \eqdef \pr_2^{\Tg B}\,\Tg(\psi_a)\,\bar V \,\psi_a^{-1}: B^1 \mapsto X, \, u \mapsto \i \sigma_a(u)$, $a\in\setZ$, where $\pr_2^{\Tg B}: B^1 \times X \mapsto X$ is the canonical projection and $\sigma_a: X^1 \mapsto X$ is the $\scale^0$ Fourier multiplier with coefficients
	\begin{equation}
		(\widehat{\sigma_a})_n = \begin{cases}
			a^2 - n^2		&\txt{if}~n<a \\
			a^2 - (n+1)^2	&\txt{if}~n\geq a \,.
		\end{cases}
	\end{equation}
	In the second place, we can define the Hamiltonian function
	\begin{equation}
		\bar h: M^1 \mapsto \setR, \, u \mapsto \frac{1}{2} \frac{\|u_x\|_0^2}{\|u\|_0^2} \,.
	\end{equation}
	This is a densely\hyp{}defined strongly $\scale$\hyp{}smooth map with derivative locally given by $\D({\bar h}_a): B^1 \mapsto X^*, \, x \mapsto - \langle \sigma_a(x), \cdot \rangle_0$, where ${\bar h}_a \eqdef \bar h\circ \psi_a^{-1}$. Finally, we conclude that if $\omega = \langle \i \cdot, \cdot \rangle_0$ is the standard symplectic structure on $X$, $\bar h$ generates $\bar V$.
\end{example}

\begin{remark}
	Analogously to the finite\hyp{}dimensional case, \thref{ex:reduced_free_schrodinger} can be interpreted in the trend of symplectic reduction. Consider the standard action of $S^1$ on $X_{\geq 0}$ given by pointwise multiplication. Defining the ($S^1$\hyp{}invariant) momentum map $\mu: X_{\geq 0} \mapsto \setR$
\begin{equation}
	\mu(x) = \frac{1}{2}(1-\|x\|_0^2) \,,
\end{equation}
we see that the action is free on $\mu^{-1}(0) = \sphere(X_0) \eqdef \{x \in X_0: \|x\|_0 = 1\}$, and that the projective Hilbert space is
\begin{equation}
	M \iso \faktor{\mu^{-1}(0)}{S^1}, \, [u] \mapsto \left[\frac{u}{\|u\|_0}\right] \,.
\end{equation}
The Hamiltonian $h: X_{\geq 0}^1 \mapsto \setR$ of \thref{ex:ham_free_schrodinger} is $S^1$\hyp{}invariant and descends to the Hamiltonian $\bar h: M^1 \mapsto \setR$ of \thref{ex:reduced_free_schrodinger} under this identification. Similarly, the flow $\varphi: \setR \times X_{\geq 0} \mapsto X_{\geq 0}$ of \thref{ex:flow_free_schrodinger} descends to $\bar\varphi: \setR \times M \mapsto M$ and, as seen in \thref{ex:reduced_free_schrodinger}, $\bar\varphi$ is generated by $\bar h$.
\end{remark}


\end{document}